 \documentclass[12pt]{amsart}
 \usepackage{amsmath,txfonts}
 \usepackage[colorlinks=true,urlcolor=blue,
 citecolor=red,linkcolor=blue,linktocpage,pdfpagelabels,
 bookmarksnumbered,bookmarksopen]{hyperref}
 \usepackage{epsfig,graphics,color}
 \numberwithin{equation}{section}
 \newtheorem{theorem}{Theorem}[section]
 \newtheorem{remark}[theorem]{Remark}
 \newtheorem{lemma}[theorem]{Lemma}
 \newtheorem{proposition}[theorem]{Proposition}
 \newtheorem{corollary}[theorem]{Corollary}

\usepackage{amssymb,amsmath}
\usepackage{epsfig,graphics}

\begin{document}

\title[ Schr\"odinger equations with potentials vanishing at infinity]
{Pohozaev manifold constraint for solving nonlinear Schr\"odinger equations with potentials vanishing at infinity}

\author[L.\ A.\ Maia, R. Ruviaro]{{Liliane de Almeida Maia*,  Ricardo Ruviaro}\\
{\tiny Departamento de Matem\'atica, UnB, 70910-900 Bras\'ilia - DF, Brazil. }\\
{\tiny \texttt{E-mail:  lilimaia.unb@gmail.com and ruviaro@mat.unb.br}}\vspace{.5cm}\\
}

\author[G. S. Pina]{{Gilberto da Silva Pina}\\
{\tiny Universidade Federal do Recôncavo da Bahia, UFRB, 44.380-000 Cruz das Almas - BA, Brazil}\\
{\tiny \texttt{E-mail: gilbertospfc@ufrb.edu.br}}
}
\thanks{{* Corresponding author}\\
The first and second authors were partially supported by FAPDF, CNPq and CAPES, Brazil. }
\maketitle

\begin{abstract}
Existence of a positive solution for a class of 
nonlinear Schr\"odinger equations with potentials which decay to zero at infinity, with an appropriate rate, approaching zero mass type limit scalar field equations, is 
established via a new composition of two translated and dilated solitons and its projection on the so called Pohozaev manifold.
\end{abstract} 

\smallskip \noindent {\small {\textbf{%
Keywords:} Nonlinear Schr\"odinger, Pohozaev manifold, Constrained minimization, Barycenter.}}

\section{Introduction}

This work deals with the existence of a positive solution for the problem 
\begin{equation}\label{1.1}
	-\Delta u+V(x)u=f(u), \qquad u \in \mathcal{D}^{1,2}(\mathbb{R}^N), \ N\geq 3,
\end{equation}
with a potential $V$ vanishing at infinity, possibly changing sign, and a nonlinearity $f$ under very mild hypotheses, asymptotically linear or superlinear and subcritical at infinity, not sa-\ tisfying any monotonicity condition.
Our goal is to investigate whether a positive ground state solution exists and, if not, to find a positive bound state, trying to loosen the assumptions found in the literature, either in the potential or in the nonlinearity \cite{amma, Azz-Pom, balio, Beres-Lions, Je-Tan}. We avoid, for instance, to apply the spectral theory approach or the so called Nehari manifold constrained approach. Under this scope, it is reasonable to look for solutions of equation \eqref{1.1} constrained to a subset of functions which satisfy Pohozaev identity \cite{pohozaev}, since all solutions do so. Moreover, combining two copies of translated and dilated positive soliton solutions of the limit zero mass scalar field equation, projecting their sum onto the so called Pohozaev manifold and studying their interaction, we are able to find a positive bound state solution in case a ground state solution does not exist. To the best of our knowledge, this is the first time that such a construction is put up when searching for a solution of a differential equation. This new approach allows us to tackle a model problem like
\[
-\Delta u+\frac{1}{(1+|x|)^k}u=\frac{2\, u^{11}-4{\sqrt 2} \, u^9+4 \, u^7}{u^{10}+1}, \quad u>0,\qquad u \in \mathcal{D}^{1,2}(\mathbb{R}^3),
\]
where $k>2$ and $f(s):=(2s^{11}-4{\sqrt 2}s^9+4s^7)/(s^{10}+1)$ is asymptotically linear at infinity, but is such that $f(s)/s$ is not increasing for $s > 0$, for instance. Likewise, $f(s)= s^7(1 - sin(s))/(1+s^4)$, for $s>0$, in $\mathbb{R}^3$ is super linear and subcritical at infinity and satisfies mild hypotheses but no monotonicity condition on $f(s)/s$.
The seminal works of Bahri and Li \cite{bali} and Cerami and Passaseo \cite{cepa95} presented constructions of bound state solutions, whenever the minimal action of the associated functional is not attained. They succeeded by building a convex combination of two soliton positive solutions of a limit problem (bumps) and projecting on the sphere of radius one in an $L^p$ space, for a pure power nonlinearity $f(s)= s^{p-1}$, with $2 < p <2^*$. Their method was applied in many works that followed and in different scenarios, but it would de hard to list them all; we would refer to \cite{Ce} and references therein.
More recently, a similar approach was developed to construct bound state solutions by using projections of convex combinations of two positive bumps on the Nehari manifold, see \cite{ Clapp-Maia, evequoz-weth,  Khatib-Maia, MaPe} and their references. The limitation, in this case, is having to assume some monotonicity on $f(s)/s$.
The novelty of our arguments, which allows to use the Pohozaev manifold as a natural constraint, is that we are able to prove in Lemma \eqref{lema4.13} that any bounded Palais Smale sequence of the associated functional restricted to this manifold is a Palais Smale sequence for the functional in the whole space (free). This has been a core issue in many previous works which applied similar constraints and required sofisticated arguments \cite{Bartsch-Soave, Bartsch-Soave2, Lehrer-Maia, Maia-Ruviaro, Mederski}. We think that our approach is somehow simpler and could be used in a large range of problems.
\\
Let $S$ be the best constant of Gagliardo-Nirenberg-Sobolev inequality 
\begin{equation}\label{desigual}
S\!\left(\displaystyle \int_{\mathbb{R}^N}\!|u|^{2^*}dx\right)^{\!\!2/2^*} \!\leq \displaystyle \int_{\mathbb{R}^N}\!|\nabla u|^2dx
\end{equation}
 for all $u \in \mathcal{D}^{1,2}(\mathbb{R}^N)$, with $2^*:=2N/(N-2)$. We will assume the following conditions on the potential $V$:\\
($V_1$)\ $V \in C^2(\mathbb{R}^N)$ and $\displaystyle\int_{\mathbb{R}^N}|V^-|^{N/2} < S^{N/2}$, where $V^-(x):=\min\{0,V(x)\}$;\\
 ($V_2$) \ There exist constants $A_0,A_1>0$ and $k\in\mathbb{R}$, $ k>\max\{2, N-2\}$ such that $$|V(x)| \leq A_0(1+|x|)^{-k} \quad \textrm{and} \quad |\nabla V(x) \cdot x| \leq A_1(1+|x|)^{-k},$$ for all $x \in \mathbb{R}^N$;\\
($V_3$) \ $\displaystyle\int_{\mathbb{R}^N}|W^+|^{N/2} < S^{N/2}$, where $W^+(x):=\max\{0,\nabla V(x) \cdot x\}$;\\
($V_4$) \ $\displaystyle\int_{\mathbb{R}^N}|Z^-|^{N/2} < \left(\frac{S}{2^*}\right)^{\!N/2}$, where $Z^-(x):=\min\!\left\{0,\dfrac{\nabla V(x) \cdot x}{N} + V(x)\right\}$;\\
($V_5$) \ $xH(x)x \in L^{N/2}(\mathbb{R}^N)$, $\displaystyle\lim_{|x| \to \infty}xH(x)x = 0$ and $\displaystyle\int_{\mathbb{R}^N}|K^+|^{N/2} < \left(\frac{2S}{2^*}\right)^{\!\!N/2}$, where $H$ denotes the Hessian matrix of $V$ and $K^+(x):=\max\!\left\{0,\nabla V(x) \cdot x + \dfrac{xH(x)x}{N}\right\}$.
\\
Moreover, considering  $F(s)=\int_{0}^{s}f(t)dt,$ we will assume the following hypotheses on the function $f$: \\
\vspace{.2cm}
($f_1$) \ $f \in C^1([0,\infty))\cap C^3((0, \infty))$, $f(s) \geq 0$ and there exists a constant $A_2>0$ such that $$\left|f^{(i)}(s)\right| \leq A_2|s|^{2^*-(i+1)},$$ 
\vspace{.2cm}where $f^{(-1)}:=F$ and $f^{(i)}$ is the $i-$th derivative of $f$, $i=0,1,2,3$;\\ 
\vspace{.2cm}
 ($f_2$) \ $\displaystyle\lim_{s \to 0^+}\dfrac{f(s)}{s^{2^*-1}} = \displaystyle\lim_{s \to +\infty}\dfrac{f(s)}{s^{2^*-1}} = 0$ \ and \ $\displaystyle\lim_{s \to +\infty}\dfrac{f(s)}{s} \geq \ell$, with $\ell \in \mathbb{R}^+$;\\ 
 \vspace{.2cm}

 ($f_3$) \ The function $$g(s):={sf'(s)}/{f(s)}$$ is non-increasing  on $(0,\gamma)$, where the constant $\gamma >0$ is defined by $\gamma=\min \{s>0; f(s)=0\}$,
 ($\gamma= \infty$ if $f(s)>0$ for all $s>0$) and $\displaystyle \lim_{s \to +\infty}g(s)<2^*-1<\displaystyle \lim_{s \to 0}g(s)$.
 
Note that $F(0)=0$ and by $(f_1)$, $(f_2)$, $F(s) \geq 0$ for $s>0$. 

Under assumptions  ($f_1$) and ($f_2$), the limit problem at infinity
\begin{equation}\label{P_0}
-\Delta u=f(u),\qquad u \in \mathcal{D}^{1,2}(\mathbb{R}^{N}),\tag{$\wp
	_0$}
\end{equation}
has a ground state solution $w$ which is positive, radially symmetric and decreasing in the radial direction, see \cite{Beres-Lions} and \cite{Mederski}.\\
Flucher in \cite[Theorem 6.5]{Flucher} and more recently V\'etois in \cite{Vetois} have shown that under $(f_1)$ there exist constants $A_4,A_5,A_6>0$ such that 
\begin{equation}\label{3.1}
A_4(1+|x|)^{-(N-2)} \leq w(x) \leq A_5(1+|x|)^{-(N-2)},
\end{equation}
\begin{equation}\label{3.2}
|\nabla w(x)| \leq A_6(1+|x|)^{-(N-1)}.
\end{equation}
A radial solution with decay \eqref{3.1} is called a fast decay solution of equation \eqref{P_0}. It is shown in \cite[Theorem 2]{Tang} and \cite[Chapter 6]{Flucher} that, in this setting with $(f_1)$, $(f_2)$ and $(f_3)$, the fast decay positive solution $w$ is unique. Moreover, Tang in \cite{Tang} obtained that $\Vert w \Vert_\infty < \gamma$. Any other hypothesis which could guarantee uniqueness of the ground state solution would suffice and possibly replace $(f_3)$. But in general this is a delicate and difficult issue. We note that the nonlinearities of the examples presented before satisfy this condition.

Now we can state our main result of existence of a solution.
\begin{theorem}\label{teo1.1}
	Assume that ($V_1$)--($V_5$) and ($f_1$)--($f_3$) hold true. Then, problem $(\ref{1.1})$ has a positive solution $u \in \mathcal{D}^{1,2}(\mathbb{R}^N)$.
\end{theorem}

\begin{remark}\label{obs1.2}
	The condition {\bf ($V_2$)} implies that $V \in L^{N/2}(\mathbb{R}^N)$ and $\nabla V(x) \cdot x \in L^{N/2}(\mathbb{R}^N)$, for all $x \in \mathbb{R}^N$. Moreover,
	\begin{equation}\label{1.2}
	V(x) \to 0, \quad \textrm{as} \; |x| \to \infty,
	\end{equation}
	\begin{equation}\label{1.3}
	\nabla V(x) \cdot x \to 0, \quad \textrm{as} \; |x| \to \infty.
	\end{equation}
	
	Condition {\bf ($V_5$)} implies that there is a constant $A_3>0$ such that
	\begin{equation}\label{1.4}
	|xH(x)x| \leq A_3, \quad  \forall \, x \in \mathbb{R}^N.
	\end{equation}
\end{remark}

Note that a model potential $V$, defined by $V(x):=(1+|x|)^{-k}$, with 	$k > \max\{2,N-2\}$ , satisfies the assumptions {\bf ($V_1$)--($V_5$)}.

Also note that assumption $(f_1)$ implies that $f^{(i)}(0)=0$ and extends $f^{(i)}$ continuously to $0$, for $i=1,2,3$.
Furthermore, L'H\^opital's rule and $(f_2)$ give that
\begin{equation} \label{hopital}
\lim_{s \to 0^+}\dfrac{f(s)}{s^{2^*-1}} = 
\lim_{s \to 0^+}\dfrac{f^{(i)}(s)}{s^{2^*-\{1+i\}} }=0, \quad i=1, 2, 3.
\end{equation}
On the other hand, hypotheses $(f_1)$ and $(f_2)$ imply
\begin{equation} \label{hop}
\lim_{s \to 0^+}\dfrac{F(s)}{s^{2^*}} = 
\lim_{s \to + \infty}\dfrac{F(s)}{s^{2^*}}=0.
\end{equation}

This paper is organized as follows: Section 2 is devoted to presenting the variational setup and the properties of the associated Pohozaev manifold. In Section 3 we study the behaviour of constrained minimizing sequences of the operator associated with
problem \eqref {1.1}. Tight estimates of interactions of two translated and dilated copies of a positive solution of the autonomous problem are obtained in Section 4. Finally, these estimates are applied in the proof of the main result of existence of a positive solution stated in Theorem \ref{teo1.1}.

\section{Pohozaev manifold and variational setting}
The well know identity obtained by Pohozaev in \cite{pohozaev} has since then been very useful as a constraint in the study of scalar field equations. We will take it as a fundamental tool for our approach. Its version for non-autonomous problems is based in the work of De Figueiredo, Lions and Nussbaum \cite{djairo} which we state here for the sake of completeness.
\begin{proposition}\label{prop2.1}
	Let $u \in \mathcal{D}^{1,2}\!(\mathbb{R}^N)\,\backslash\,\{0\}$ be a solution of problem
$	-\Delta u=g(x,u)$, $x \in \Omega $, 
$	u(x)=0$, $ x \in \partial\Omega $,
	where $\Omega \subset \mathbb{R}^N$ is a regular domain in $\mathbb{R}^N$ and $g \in C(\overline{\Omega}\times \mathbb{R},\mathbb{R})$. If $G(x,u)=\int_{0}^{u}g(x,s)ds$ is such that $G(\cdot,u(\cdot))$ and $x_iG_{x_i}(\cdot,u(\cdot))$ are in $L^1(\Omega)$, then $u$ satisfies $$N\!\int_{\Omega}G(x,u)dx+\sum_{i=1}^{N}\int_{\Omega}x_iG_{x_i}(x,u)dx-\frac{N-2}{2}\!\int_{\Omega}|\nabla u|^2dx=\frac{1}{2}\int_{\partial\Omega}|\nabla u|^2x\cdot \eta(x)dS_x,$$ where $\eta$ denotes the unitary exterior normal vector to boundary $\partial\Omega$ and $dS_x$ represents the area element $(N-1)$-dimensional of $\partial\Omega$. Moreover, if $\Omega=\mathbb{R}^N$, then
	\begin{equation}\label{2.1}
	\frac{N-2}{2}\!\int_{\mathbb{R}^N}|\nabla u|^2dx=N\!\int_{\mathbb{R}^N}G(x,u)dx+\sum_{i=1}^{N}\int_{\mathbb{R}^N}x_iG_{x_i}(x,u)dx.
	\end{equation}
\end{proposition}

In the case of problem (\ref{1.1}), by (\ref{2.1}), we have the following Pohozaev identity
\begin{equation}\label{2.5}
\frac{N-2}{2}\!\int_{\mathbb{R}^N}|\nabla u|^2dx=N\!\int_{\mathbb{R}^N}G(x,u)dx-\frac{1}{2}\int_{\mathbb{R}^N}\nabla V(x)\cdot x\,u^2dx,
\end{equation}
where $G(x,u):=F(u)-V(x)\dfrac{u^2}{2}$.

Let us consider the functional $J_V:\mathcal{D}^{1,2}(\mathbb{R}^N) \to \mathbb{R}$, defined by 
\begin{eqnarray*}
	J_V(u)
		= \frac{N-2}{2}\!\int_{\mathbb{R}^N}|\nabla u|^2dx+\frac{N}{2}\!\int_{\mathbb{R}^N}\left(\frac{\nabla V(x)\cdot x}{N}+V(x)\right)u^2dx-N\!\int_{\mathbb{R}^N}F(u)dx,
\end{eqnarray*}
and define the Pohozaev manifold associated to the problem (\ref{1.1}) by $$\mathcal{P}_V:=\{u \in \mathcal{D}^{1,2}(\mathbb{R}^N)\,\backslash\,\{0\}: J_V(u)=0\}.$$
Let us also consider the Pohozaev manifold $\mathcal{P}_0$ associated to the limit problem (\ref{P_0}). We have $$\mathcal{P}_0:=\{u \in \mathcal{D}^{1,2}(\mathbb{R}^N)\,\backslash\,\{0\}: J_0(u)=0\},$$ where $$J_0(u):=\frac{N-2}{2}\!\int_{\mathbb{R}^N}|\nabla u|^2dx-N\!\int_{\mathbb{R}^N}F(u)dx.$$ 

We define $f(s):=-f(-s)$ for $s<0$. Then, by condition {\bf ($f_1$)}, we have $f \in C^1(\mathbb{R})$ and it is an odd function. Note that, if $u$ is a positive solution of problem (\ref{1.1}) for this new function, it is also a solution of (\ref{1.1}) for the original function $f$. Hereafter, we shall consider this extension, and establish the existence of a positive solution for \eqref{1.1}. 

We consider the Hilbert space $\mathcal{D}^{1,2}(\mathbb{R}^N):=\{u \in L^{2^*}(\mathbb{R}^N): \nabla u \in L^2(\mathbb{R}^N)\}$ with its standard scalar product and norm $$\langle u,v\rangle:=\displaystyle\int_{\mathbb{R}^N}\nabla u \cdot \nabla v\,dx, \qquad \Vert u\Vert^2:=\displaystyle\int_{\mathbb{R}^N}|\nabla u|^2dx.$$

Since $f \in C^1(\mathbb{R})$ and $f$ satisfies {\bf ($f_1$)}, a classical result of Berestycki and Lions establishes the existence of a ground state solution $w \in C^2(\mathbb{R}^N)$ to problem (\ref{P_0}), which is positive, radially symmetric and decreasing in the radial direction, see \cite[Theorem 4]{Beres-Lions}.

In what follows, we will use the following notation:
given $u,v \in \mathcal{D}^{1,2}\!(\mathbb{R}^N)$, let's define
\begin{equation}\label{3.3}
\langle u,v\rangle_V:=\displaystyle\int_{\mathbb{R}^N}\nabla u\cdot \nabla v+V(x)uv\;\!dx, \qquad \Vert u\Vert_{V}^{2}:=\displaystyle\int_{\mathbb{R}^N}|\nabla u|^2+V(x)u^2\;\!dx;
\end{equation}
also $\Vert\cdot\Vert_q$ denotes the $L^q(\mathbb{R}^N)$-norm, for all $q \in [1,\infty)$ and $C$, $C_i$ are positive constants which may vary from line to line. 
By assumptions {\bf ($V_1$)} and {\bf ($V_2$)}, we can see that the expressions in (\ref{3.3}) are well defined and, using the Sobolev inequality, we conclude that $\Vert\cdot\Vert_V$ is a norm in $\mathcal{D}^{1,2}(\mathbb{R}^N)$ which is equivalent to the standard one.

Because of assumption $(f_1)$ the functional $I_V:\mathcal{D}^{1,2}(\mathbb{R}^N) \to \mathbb{R}$, defined by 
$$I_V(u)=\frac{1}{2}\Vert u\Vert_{V}^{2}-\int_{\mathbb{R}^N}\!F(u)dx$$
 is $C^1$ and hence weak solutions of problem (\ref{1.1}) are its critical points.

We also have that solutions of (\ref{P_0}) are critical points of the functional $$I_0(u):=\frac{1}{2}\Vert u\Vert^2-\int_{\mathbb{R}^N}F(u)dx, \qquad u \in \mathcal{D}^{1,2}(\mathbb{R}^N).$$

We recall that $w$ is a ground state solution of the limit problem (\ref{P_0}) if
\begin{eqnarray}\label{m} 
I_0(w)=m:=\inf\{I_0(u): u\in \mathcal{D}^{1,2}(\mathbb{R}^N)\setminus \{0\} \;\text{is a solution of the limit problem}  \}.
\end{eqnarray}

Later on we will show that $\mathcal{P}_V \neq \emptyset$ and that $p_V \in (0,p_0]$, where 
\begin{equation}\label{3.5}
p_V=\inf_{u \in \mathcal{P}_V}I_V(u), \qquad  p_0=\inf_{u \in \mathcal{P}_0}I_0(u).
\end{equation}
It was shown in \cite{Mederski} that $m=p_0$, under more general hypotheses, which contains ours as a particular case. 

	The following result is an essential tool for developing our new arguments.
\begin{lemma}\label{lema4.3}
	Assume that {\bf ($V_1$)}, {\bf ($V_4$)} and {\bf ($f_1$)} hold true. Then, there exists a real number $\rho > 0$ such that $\inf_{u \in \mathcal{P}_V}\Vert \nabla u \Vert_2 \geq \rho$.
\end{lemma}
\begin{proof}
	Let $u \in \mathcal{P}_V$ be given. Then, using $(\ref{2.5})$, H\"older inequality and hypotheses {\bf ($V_4$)} and {\bf ($f_1$)}, we get
\begin{eqnarray*}
		\int_{\mathbb{R}^N}F(u)dx 
		&\geq& \frac{1}{2^*}\!\int_{\mathbb{R}^N}|\nabla u|^2dx - \frac{1}{2 \cdot 2^*}\!\int_{\mathbb{R}^N}|\nabla u|^2dx 
		= \frac{1}{2 \cdot 2^*}\!\int_{\mathbb{R}^N}|\nabla u|^2dx > 0.
	\end{eqnarray*}
	Thus,
	\begin{equation}\label{4.1}
	\Vert \nabla u \Vert_2^2 = \int_{\mathbb{R}^N}|\nabla u|^2dx \leq 2 \cdot 2^*\!\int_{\mathbb{R}^N}F(u)dx \leq 2 \cdot 2^*A_2\!\int_{\mathbb{R}^N}|u|^{2^*}dx = 2 \cdot 2^*A_2\Vert u \Vert_{2^*}^{2^*}.
	\end{equation}
	On the other hand, using the Gagliardo-Nirenberg-Sobolev inequality $(\ref{desigual})$ 
	in (\ref{4.1}), it follows
	$$0 < \Vert \nabla u \Vert_2^2 \leq 2 \cdot 2^*A_2\!\left(\!\frac{1}{\sqrt{S}}\Vert \nabla u \Vert_2\!\right)^{\!\!2^*}.$$ Therefore, taking  $\rho^{2^*-2}=\dfrac{S^{2^*\!/2}}{2 \cdot 2^*A_2}$, we obtain $\Vert \nabla u \Vert_2 \geq \rho$. Since $u \in \mathcal{P}_V$ is arbitrary, it follows that $\inf_{u \in \mathcal{P}_V}\Vert \nabla u \Vert_2 \geq \rho$.
\end{proof}

	\begin{lemma}\label{coracao2}
	Let $u\in \mathcal{P}_V$, then $J'_V(u)\neq 0$ in $(\mathcal{D}^{1,2}(\mathbb{R}^N))'$.
\end{lemma}
\begin{proof}
Suppose by contradiction that $J'_V(u)=0$ in $(\mathcal{D}^{1,2}(\mathbb{R}^N))'$, then $u$ is a critical point of the functional $J_V\in C^1$, hence it is a weak solution of the equation
$$-(N-2)\Delta u+N\left(\frac{\nabla V(x).x}{N}+V(x)  \right)u=Nf(u), \qquad u \in \mathcal{D}^{1,2}(\mathbb{R}^{N}).$$
By Proposition 	\ref{prop2.1},  $u$ satisfies the Pohozaev identity for this equation, as follows
$$\frac{(N-2)^2}{2}\int_{\mathbb{R}^N}|\nabla u|^2=N^2\int_{\mathbb{R}^N}F(u)-\left(\frac{\nabla V(x).x}{N}+V(x) \right)\frac{u^2}{2}-\frac{N}{2}\int_{\mathbb{R}^N}\left(\frac{xH(x)x}{N}+\nabla V(x).x \right) u^2.$$
Using that $J_V(u)=0$, it yields
$$\frac{(N-2)^2}{2}\int_{\mathbb{R}^N}|\nabla u|^2= \frac{(N-2)N}{2}\int_{\mathbb{R}^N}|\nabla u|^2-\frac{N}{2}\int_{\mathbb{R}^N}\left(\frac{xHx}{N}+\nabla V(x).x   \right)u^2,$$
or equivalently
\begin{equation}\label{desigual1}
0=(N-2)\int_{\mathbb{R}^N}|\nabla u|^2-\frac{N}{2}\int_{\mathbb{R}}\left(\frac{xHx}{N}+ \nabla V(x).x \right)u^2.
\end{equation}
Assumption $(V_5)$ implies
$$\frac{N}{2}\int_{\mathbb{R}^N}\left(\frac{xHx}{N}+\nabla V(x).x   \right)u^2 < \frac{N}{2}\int_{\mathbb{R}^N}K^+(x)u^2\leq \frac{N-2}{2}\int_{\mathbb{R}^N}|\nabla u|^2.$$
Hence, substituting this inequality in $(\ref{desigual1})$ and by Lemma \ref{lema4.3} it results that
$$0> (N-2)\int_{\mathbb{R}^N}|\nabla u|^2-\frac{N-2}{2}\int_{\mathbb{R}^N}|\nabla u|^2\geq\frac{N-2}{2}\rho^2 > 0,$$
which is an absurd. 
\end{proof}	

Next proposition states the main properties of $\mathcal{P}_V$. These results are by now standard, but we will include short proofs just to enlighten the small differences.
\begin{proposition}\label{lema3.1}
	Assume that {\bf ($V_1$)}, {\bf ($V_4$)}, {\bf ($f_1$)} and {\bf ($f_2$)} hold true. Then:\\
	$(a)$ there exists $\varrho>0$ such that $\Vert u\Vert_{V}\geq\varrho$,
		for every $u\in\mathcal{P}_{V}$;\\
	$(b)$ $u_0 \equiv 0$ is an isolated point of $J_V^{-1}\!\left(\left\{0\right\}\right)$;\\
	$(c)$ $\mathcal{P}_{V}$ is a closed $C^{2}$-submanifold of $\mathcal{D}^{1,2}(\mathbb{R}^{N})$.		
\end{proposition}

\begin{proof} 
	(a) Let $u \in \mathcal{D}^{1,2}(\mathbb{R}^N)\,\backslash\,\{0\}$ be given arbitrarily. Using conditions {\bf ($V_4$)}, {\bf ($f_1$)} and H\"older inequality, we have 
	\begin{eqnarray*}
		J_V(u) &=& \frac{N-2}{2}\!\int_{\mathbb{R}^N}|\nabla u|^2dx+\frac{N}{2}\!\int_{\mathbb{R}^N}\!\left(\frac{\nabla V(x)\cdot x}{N}+V(x)\!\right)\!u^2dx-N\!\int_{\mathbb{R}^N}F(u)dx \\
	&\geq& \frac{N-2}{4}\!\int_{\mathbb{R}^N}|\nabla u|^2dx - A_1N\!\int_{\mathbb{R}^N}\left|u\right|^{2^*}\!dx.
	\end{eqnarray*}
	By the equivalence of norms $\Vert \cdot \Vert_V$ and $\Vert \cdot \Vert$ in $\mathcal{D}^{1,2}(\mathbb{R}^N)$ and by the continuity of the embedding of  $\mathcal{D}^{1,2}(\mathbb{R}^N)$ into $L^{2^*}(\mathbb{R}^N)$, we obtain $C_1,C_2>0$ such that $$\int_{\mathbb{R}^N}|\nabla u|^2dx \geq C_1\Vert u \Vert_V^2, \qquad \int_{\mathbb{R}^N}\left|u\right|^{2^*}\!dx \leq C_2\Vert u \Vert_{V}^{2^*}.$$ 
	If $u \in \mathcal{P}_V$, then $J_V(u) = 0$, so this implies
	\begin{eqnarray*}
		0 \geq \frac{N-2}{4}\!\int_{\mathbb{R}^N}|\nabla u|^2dx - A_1N\!\int_{\mathbb{R}^N}\left|u\right|^{2^*}\!dx 
	= \left[C_1\,\frac{N-2}{2} - C_2\,A_1\,N\Vert u \Vert_{V}^{2^*-2}\right]\!\Vert u \Vert_V^2,
	\end{eqnarray*}
	that is, $\Vert u \Vert_{V}^{2^*-2} \geq \frac{C_1(N-2)}{2C_2\,A_1N}$. Therefore, taking $\varrho:=\frac{C_1(N-2)}{2C_2\,A_1N}^{\!1/(2^*-2)}$, we have $\Vert u\Vert_V \geq \varrho$, for every $u \in \mathcal{P}_V$, proving item (a).
	\\
	(b) It follows from (a) that exists $\varrho > 0$ such that if $0 < \Vert u \Vert_V < \varrho$, then $J_V(u) > 0$. Therefore, $u_0 \equiv 0$ is an isolated point of $J_V^{-1}\!\left(\left\{0\right\}\right)$, proving item (b).
	\\
	(c) Observe that $J_V$ is continuous, and so $\mathcal{P}_V \cup \{0\}=J_V^{-1}(\{0\})$ is a closed subset of $\mathcal{D}^{1,2}(\mathbb{R}^N)$. By item (b), we get that $\mathcal{P}_V$ is closed in $\mathcal{D}^{1,2}(\mathbb{R}^N)$. 
	Moreover, by Lemma \ref{coracao2}, it holds that $J'_V(u)\neq 0$ 
	for every $u \in \mathcal{P}_V$, which implies that 0 is a regular value of $J_V:\mathcal{D}^{1,2}(\mathbb{R}^N)\,\backslash\,\{0\} \to \mathbb{R}$. So, as $J_V \in C^2\left(\mathcal{D}^{1,2}(\mathbb{R}^N),\mathbb{R}\right)$, it follows that $\mathcal{P}_V$ is a closed $C^2$-submanifold of $\mathcal{D}^{1,2}(\mathbb{R}^N)$.
\end{proof}

\begin{lemma}\label{lema3.2}
	Assume {\bf ($V_1$)} and {\bf ($V_3$)} hold true. Given a positive constant $d$ and a sequence $(u_n) \subset \mathcal{P}_V$ such that $I_V(u_n) \to d$, then the sequence $(u_n)$ is bounded in $\mathcal{D}^{1,2}\!(\mathbb{R}^N)$.
\end{lemma}

\begin{proof}
	Using the definition of $\mathcal{P}_V$, the hypotheses {\bf ($V_1$)}, {\bf ($V_3$)}, H\"older inequality and the equi-valence of norms $\Vert \cdot \Vert$ and $\Vert \cdot \Vert_V$, there exists a constant $C_1>0$ such that
	\begin{eqnarray*}
		I_V(u_n)
		\geq \frac{1}{N}\!\int_{\mathbb{R}^N}|\nabla u_n|^2dx - \frac{1}{2N}\int_{\mathbb{R}^N}|\nabla u_n|^2dx 
		= \frac{1}{2N}\int_{\mathbb{R}^N}|\nabla u_n|^2dx \geq C_1\Vert u_n \Vert_V^2.
	\end{eqnarray*}
Since $I_V(u_n) \to d>0$, there exists $n_0 \in\mathbb{N}$ such that $d+1 \geq I_V(u_n) \geq C_1\Vert u_n \Vert_V^2, \forall \, n \geq n_0.$ Thus, taking $C_2:=\max\!\left\{\Vert u_1 \Vert_V, \cdots, \Vert u_{n_0} \Vert_V, \sqrt{\frac{d+1}{C_1}}\right\}$, it follows that $$\Vert u_n \Vert_V \leq C_2, \quad \forall \, n \in \mathbb{N}.$$ Therefore, $(u_n)$ is bounded in $\mathcal{D}^{1,2}\!(\mathbb{R}^N)$ and the proof is complete.
\end{proof}

\section{Study of energy levels}

Let us gather some information about the energy levels of $I_V$ and $I_0$.

\begin{lemma}\label{lema4.4}
	Assume {\bf ($V_1$)}, {\bf ($V_3$)}, {\bf ($V_4$)} and {\bf ($f_1$)} hold true. Then, $p_V>0$.
\end{lemma}

\begin{proof}
	Let $u \in \mathcal{P}_V$, then by $(\ref{2.5})$
	\begin{eqnarray*}
		\frac{N-2}{2}\!\int_{\mathbb{R}^N}|\nabla u|^2dx &=& N\!\int_{\mathbb{R}^N}F(u)dx-\frac{N}{2}\!\int_{\mathbb{R}^N}\left(\frac{\nabla V(x) \cdot x}{N}+V(x)\!\right)\,\!u^2dx.
	\end{eqnarray*} 
	H\"older inequality, condition {\bf ($V_3$)} and the constant $\rho > 0$ obtained in Lemma \ref{lema4.3}, yield
	\begin{eqnarray*}
		I_V(u) 
		&\geq& \frac{1}{N}\!\int_{\mathbb{R}^N}|\nabla u|^2dx - \frac{1}{2N}\!\left(\int_{\mathbb{R}^N}\left|W^+(x)\right|^{N/2}dx\right)^{2/N}\!\left(\int_{\mathbb{R}^N}\left|u^2\right|^{2^*/2}dx\right)^{2/2^*} \\
	&\geq& \frac{1}{2N}\!\int_{\mathbb{R}^N}|\nabla u|^2dx \geq \frac{1}{2N}\rho^2 > 0.
	\end{eqnarray*}
	Since $u \in \mathcal{P}_V$ was taken arbitrarily, it follows that $p_V > 0$.
\end{proof}

In the next lemmas we will prove that $p_V \leq c_V \leq c_0 = m = p_0$, where $m$ is defined by (\ref{m}) and
\begin{equation}\label{c_0}
c_0:=\displaystyle\inf_{\gamma \, \in \, \Gamma_0}\displaystyle\max_{0 \, \leq \, t \, \leq \, 1}I_0(\gamma(t)), \quad \Gamma_0:=\left\{\gamma \in C\!\left([0,1],\mathcal{D}^{1,2}\!(\mathbb{R}^N)\!\right): \gamma(0)=0, I_0(\gamma(1))<0\right\};
\end{equation}
\begin{equation}\label{c_V}
c_V:=\displaystyle\inf_{\gamma \, \in \, \Gamma_V}\displaystyle\max_{0 \, \leq \, t \, \leq \, 1}I_V(\gamma(t)), \quad \Gamma_V:=\left\{\gamma \in C\!\left([0,1],\mathcal{D}^{1,2}\!(\mathbb{R}^N)\!\right): \gamma(0)=0, I_V(\gamma(1))<0\right\}.
\end{equation}

\begin{lemma}\label{lema4.10}
	Assume that {\bf ($V_1$)} and {\bf ($f_1$)--($f_2$)} hold true. Then, $c_0 \geq c_V$.
\end{lemma}

\begin{proof}
	Let $\varepsilon>0$ be given arbitrarily. We know there exists $\gamma \in \Gamma_0$ such that $I_0(\gamma(t)) \leq c_0+\varepsilon/2$, for all $t \in [0,1]$. Consider the translation $\tau_y(\gamma(t))x:=\gamma(t)(x-y)$ for $y \in \mathbb{R}^N$ chosen, such that $|y|$ is sufficiently large. Thus, using the hypothesis {\bf ($V_1$)} and the Lebesgue dominated convergence theorem, we have
	\begin{eqnarray*}
		I_V\left(\tau_y \circ \gamma(1)\right) = I_0\left(\tau_y \circ \gamma(1)\right) + \int_{\mathbb{R}^N}V(x+y)\left(\gamma(1)\right)^2dx 
		= I_0(\gamma(1)) + o_y(1) < 0, 
	\end{eqnarray*}
	proving that $\tau_y \circ \gamma \in \Gamma_V$. Let $t_y \in [0,1]$ be such that $$I_V\left(\tau_y \circ \gamma\left(t_y\right)\right):=\max_{0 \, \leq \, t \, \leq \, 1}I_V\left(\tau_y \circ \gamma(t)\right)\;\; \text{and}\;\;\; I_V\left(\tau_y \circ \gamma\left(t_y\right)\right) \leq I_0\left(\tau_y \circ \gamma\left(t_y\right)\right) + \frac{\varepsilon}{2}.$$ Then, 
	\begin{eqnarray*}
		c_0+\varepsilon &\geq& I_0\left(\gamma\left(t_y\right)\right) + \frac{\varepsilon}{2} = I_0\left(\tau_y \circ \gamma\left(t_y\right)\right) + \frac{\varepsilon}{2} \geq I_V\left(\tau_y \circ \gamma\left(t_y\right)\right) \\
		&=& \max_{0 \, \leq \, t \, \leq \, 1}I_V\left(\tau_y \circ \gamma(t)\right) \geq \inf_{\gamma \, \in \, \Gamma_V}\max_{0 \, \leq \, t \, \leq \, 1}I_V(\gamma(t)) = c_V.
	\end{eqnarray*}
	Since $\varepsilon>0$ is arbitrary, it follows that $c_0 \geq c_V$.
\end{proof}

Now, let us present a property of intersection of $\mathcal{P}_V$ with the rescaling of the paths in the Mountain Pass Theorem \cite{Amb-Rab}.

\begin{lemma}\label{lema4.11}
	Assume that {($V_1$)--($V_4$)} and {($f_1$)-($f_2$)} hold true. Then, for every $\gamma \in \Gamma_V$, there exists $t_\gamma \in (0,1)$ such that $\gamma (t_\gamma) \in \mathcal{P}_V$. In particular, one has $c_V \geq p_V$.
\end{lemma}

\begin{proof}
	Arguing as in the proof of Proposition \ref{lema3.1}(b), we obtain $\varrho > 0$ such that if $u \in  \mathcal{D}^{1,2}(\mathbb{R}^N)$, with $0 < \Vert u \Vert_V < \varrho$, then $J_V(u) > 0$. Moreover, we observe that
	\begin{eqnarray*}
		J_V(u) 
	&\leq& NI_V(u) - \int_{\mathbb{R}^N}|\nabla u|^2dx + \frac{S}{2}\left(\int_{\mathbb{R}^N}|u|^{2^*}dx\right)^{\!2/2^*} \\
	&\leq& NI_V(u) - \frac{1}{2}\int_{\mathbb{R}^N}|\nabla u|^2dx, \qquad \forall \, u \in \mathcal{D}^{1,2}(\mathbb{R}^N).
	\end{eqnarray*}
	Hence, for each path $\gamma \in \Gamma_V$, we have $J_V(\gamma(0))=0$ and $J_V(\gamma(1)) \leq NI_V(\gamma(1))<0$. By continuity of $J_V$, there exists $t_\gamma \in (0,1)$ such that $\Vert \gamma(t_\gamma) \Vert_V \geq \varrho$ and $J_V(\gamma(t_\gamma)) = 0$, proving that $\gamma(t_\gamma) \in \mathcal{P}_V$. In particular, we have $\displaystyle\max_{0 \, \leq \, t \, \leq \, 1}I_V(\gamma(t)) \geq I_V(\gamma(t_\gamma))$ and so $c_V \geq p_V$.
\end{proof}
Now, using the previous results and some new results by Mederski in \cite{Mederski} (Theorem 1.1), we are ready to obtain the following inequality.
\begin{lemma}\label{lema4.12}
	Assume that {($V_1$)--($V_4$)} and { ($f_1$)--($f_2$)} hold true. Then, $p_V \leq p_0$.
\end{lemma}
\begin{proof}
	Indeed, using the Lemmas \ref{lema4.4}, \ref{lema4.10} and \ref{lema4.11}, we have 
	$$0 < p_V \leq c_V \leq c_0 = m = p_0.$$
\end{proof}
We write $\nabla I_V(u)$ for the gradient of $I_V$ at $u$ with respect to the scalar product $\langle \cdot,\cdot\rangle_V$, and $\nabla_{\mathcal{P}_V}I_V(u)$ for its orthogonal projection onto the tangent space of $\mathcal{P}_V$ at $u$.

Recall that a sequence $(u_n)$ in $\mathcal{D}^{1,2}\!(\mathbb{R}^N)$ is said to be a \textit{(PS)$_d$-sequence for $I_V$ with $d \in \mathbb{R}$} if $I_V(u_n) \to d$ and $I^\prime_V(u_n) \to 0$ in $(\mathcal{D}^{1,2}(\mathbb{R}^N)^\prime$. A sequence $(u_n)$ in $\mathcal{P}_V$ is a \textit{(PS)$_d$-sequence for $I_V$ restricted to $\mathcal{P}_V$} if $I_V(u_n) \to d$ and $\Vert I_V|_{\mathcal{P}_V}^\prime(u_n) \Vert_{\left(\mathcal{D}^{1,2}\left(\mathbb{R}^N\right)\right)^\prime} \to 0$ or $\Vert \nabla_{\mathcal{P}_V}I_V(u) \Vert_{\left(\mathcal{D}^{1,2}\left(\mathbb{R}^N\right)\right)^\prime} \to 0$.

\begin{lemma}\label{lema4.13}
	Assume that {($V_1$)--($V_5$)} and {($f_1$)--($f_2$)} hold true. Let $(u_n) \subset \mathcal{P}_V$ be a (PS)$_d$-sequence for $I_V$ on $\mathcal{P}_V$. Then, $(u_n)$ is a (PS)$_d$-sequence for $I_V$ (free).
\end{lemma}

\begin{proof}
	By definition, we have $I_V(u_n) \to d$ and $I_V|_{\mathcal{P}_V}'(u_n) \to 0$, i.e. $I_V'(u_n) + \lambda_nJ_V'(u_n) \to 0$, where $(\lambda_n)$ is a sequence of real numbers and $d>0$, from Lemma \eqref{lema4.4}. Let us show that $\Vert J_V'(u_n) \Vert_{\left(\mathcal{D}^{1,2}\left(\mathbb{R}^N\right)\right)'}$ is bounded and $\lambda_n \to 0$. Thus, $I_V'(u_n) \to 0$ in $(\mathcal{D}^{1,2}(\mathbb{R}^N))^\prime$ and so $(u_n)$ is (PS)$_d$-sequence for $I_V$ (free). Indeed, since $(u_n)$ is bounded, by Lemma \ref{lema3.2}, it follows that $J_V'(u_n)$ is bounded. Furthermore, $J_V'(u_n) \not = 0$ by Lemma \ref{coracao2}.
	
	Now, let us show that $\lambda_n \to 0$. By Lemma \ref{lema4.3}, there exists a constant $\rho > 0$ such that $\|\nabla u_n\|^2_2\geq \rho^2$, for every $n \in \mathbb{N}$. Since $(u_n)$ is bounded in $\mathcal{D}^{1,2}\!(\mathbb{R}^N)$, there exists $(\alpha_n) \subset \mathbb{R}$, with $\alpha_n \to 0$, such that $$I_V'(u_n)u_n + \lambda_nJ_V'(u_n)u_n = \alpha_n\!\int_{\mathbb{R}^N}|\nabla u_n|^2dx,$$ where $\alpha_n:=\dfrac{I'_V|_{\mathcal{P}_V}(u_n)u_n}{\int_{\mathbb{R}^N}|\nabla u_n|^2dx}$. That is,
	\begin{eqnarray*}
		\alpha_n\!\!\int_{\mathbb{R}^N}\!|\nabla u_n|^2dx
		&=& (1+\lambda_n (N-2))\!\!\int_{\mathbb{R}^N}|\nabla u_n|^2dx+(1+\lambda_n N)\!\int_{\mathbb{R}^N}V(x)u_n^2dx \\
		&& + \lambda_n\!\int_{\mathbb{R}^N}\nabla V(x) \cdot xu_n^2dx-(1+\lambda_n N)\!\int_{\mathbb{R}^N}f(u_n)u_n\,dx.
	\end{eqnarray*}
	Hence, we have
	\begin{eqnarray*}
		0 &=& (1-\alpha_n+\lambda_n (N-2))\!\!\int_{\mathbb{R}^N}|\nabla u_n|^2dx+(1+\lambda_n N)\!\int_{\mathbb{R}^N}V(x)u_n^2dx \\
		&& + \lambda_n\!\int_{\mathbb{R}^N}\nabla V(x) \cdot xu_n^2dx-(1+\lambda_n N)\!\int_{\mathbb{R}^N}f(u_n)u_n\,dx.
	\end{eqnarray*}
	Note that the above expression can be associated with the equation
	\begin{equation}\label{4.5}
	-(1-\alpha_n+\lambda_n (N-2))\Delta v+(1+\lambda_n N)V(x)v+\lambda_n\nabla V(x) \cdot x\;v=(1+\lambda_n N)f(v),
	\end{equation}
	$v \in \mathcal{D}^{1,2}\!(\mathbb{R}^N)$. Moreover, the solutions of the equation (\ref{4.5}) satisfy a Pohozaev identity and admit an associated Pohozaev manifold, defined by $\widetilde{J}_V^{-1}(\{0\})$, where
	\begin{eqnarray*}
		\widetilde{J}_V(v)=\frac{(1-\alpha_n+\lambda_n (N-2))(N-2)}{2}\!\!\int_{\mathbb{R}^N}|\nabla v|^2dx-N\!\!\int_{\mathbb{R}^N}\widetilde{G}(x,v)dx-\displaystyle\sum_{i=1}^{N}\!\int_{\mathbb{R}^N}\widetilde{G}_{x_i}(x,v)x_idx,
	\end{eqnarray*}
	with $$\widetilde{G}(x,v)=-\frac{1+\lambda_n N}{2}V(x)v^2-\frac{\lambda_n}{2}\nabla V(x) \cdot x\;v^2+(1+\lambda_n N)F(v)$$ and $$\displaystyle\sum_{i=1}^{N}\int_{\mathbb{R}^N}\widetilde{G}_{x_i}(x,v)x_idx=-\frac{1+\lambda_n N}{2}\!\int_{\mathbb{R}^N}\nabla V(x) \cdot x\;v^2dx-\frac{\lambda_n}{2}\!\int_{\mathbb{R}^N}xH(x)x\;v^2dx,$$ where $H(x)$ denotes the Hessian matrix of $V(x)$. Simplifying, it follows that
	\begin{eqnarray}\label{4.6}
	\widetilde{J}_V(v) \!\!\!&=&\!\!\!\! \frac{(1-\alpha_n+\lambda_n (N-2))(N-2)}{2}\!\!\int_{\mathbb{R}^N}|\nabla v|^2dx\nonumber 
	 + \frac{N(1+\lambda_n N)}{2}\!\!\int_{\mathbb{R}^N}\left[V(x)+\frac{\nabla V(x) \cdot x}{N}\right]\!v^2dx\nonumber \\
	&& + \frac{\lambda_n N}{2}\!\!\int_{\mathbb{R}^N}\!\left[\nabla V(x) \cdot x + \frac{xH(x)x}{N}\right]\!v^2dx
	 - N(1+\lambda_n N)\!\!\int_{\mathbb{R}^N}F(v)dx.
	\end{eqnarray}
	
	Making $v=u_n$ in (\ref{4.6}) and since $u_n \in \mathcal{P}_V$, we have
	\begin{equation*}
	N\!\!\int_{\mathbb{R}^N}\left[V(x)+\frac{\nabla V(x) \cdot x}{N}\right]\!u_n^2dx-2N\!\!\int_{\mathbb{R}^N}F(u_n)dx = -(N-2)\!\!\int_{\mathbb{R}^N}|\nabla u_n|^2dx,
	\end{equation*}
	and so
	\begin{eqnarray*}
		\widetilde{J}_V(u_n) &=& \frac{(1-\alpha_n+\lambda_n (N-2))(N-2)}{2}\!\!\int_{\mathbb{R}^N}|\nabla u_n|^2dx - \frac{(1+\lambda_n N)(N-2)}{2}\!\!\int_{\mathbb{R}^N}|\nabla u_n|^2dx \\ 
		&& + \frac{\lambda_n N}{2}\!\!\int_{\mathbb{R}^N}\left[\nabla V(x) \cdot x + \frac{xH(x)x}{N}\right]\!u_n^2dx \\
		&=& -\left(\!\frac{\alpha_n+2\lambda_n}{2}\!\right)\!(N-2)\!\!\int_{\mathbb{R}^N}|\nabla u_n|^2dx + \frac{\lambda_n N}{2}\!\!\int_{\mathbb{R}^N}\left[\nabla V(x) \cdot x + \frac{xH(x)x}{N}\right]\!u_n^2dx.
	\end{eqnarray*}
	On the other hand, we have that $u_n$ is a solution of the equation (\ref{4.5}), and thus $\widetilde{J}_V(u_n)=0$. Then,
	\begin{eqnarray*}
		\frac{(\alpha_n+2\lambda_n)(N-2)}{N}\!\int_{\mathbb{R}^N}|\nabla u_n|^2dx = \lambda_n\!\int_{\mathbb{R}^N}\left[\nabla V(x) \cdot x + \frac{xH(x)x}{N}\right]\!u_n^2dx,
	\end{eqnarray*}
	or equivalently,
	\begin{equation}\label{4.7}
	\alpha_n(N-2)\!\!\int_{\mathbb{R}^N}|\nabla u_n|^2dx = \lambda_n\!\left[N\!\!\int_{\mathbb{R}^N}\left(\nabla V(x) \cdot x + \frac{xH(x)x}{N}\right)\!u_n^2dx - 2(N-2)\!\!\int_{\mathbb{R}^N}|\nabla u_n|^2dx\right].
	\end{equation}
	Note that, using H\"older's inequality and hypothesis {\bf ($V_5$)}, it holds 
	\begin{eqnarray}\label{4.8}
		N\int_{\mathbb{R}^N}\left(\nabla V(x) \cdot x + \frac{xH(x)x}{N}\right)u_n^2dx {}\nonumber &\leq& N \int_{\mathbb{R}^N}K^+(x)u_n^2dx 
		< \frac{2NS}{2^*}\left(\int_{\mathbb{R}^N}|u_n|^{2^*}dx\right)^{2/2^*}\\ 
	&\leq& (N-2)\int_{\mathbb{R}^N}|\nabla u_n|^2dx.
	\end{eqnarray}
So it follows from (\ref{4.8}) that
	\begin{eqnarray*}
	N\int_{\mathbb{R}^N}\left(\nabla V(x) \cdot x + \frac{xH(x)x}{N}\!\right)u_n^2dx - 2(N-2)\int_{\mathbb{R}^N}|\nabla u_n|^2dx 
	 &<& -(N-2)\int_{\mathbb{R}^N}|\nabla u_n|^2dx\\ \leq (2-N)\rho^2 < 0,
	\end{eqnarray*}
	which means that the bracket term in (\ref{4.7}) is bounded above by a strictly negative constant. Therefore, taking $n \to \infty$ in (\ref{4.7}), it follows that $\lambda_n \to 0$, proving the claim.
\end{proof}

\begin{corollary}\label{cor4.14}
	Assume that {\bf ($V_1$)--($V_5$)} and {\bf ($f_1$)--($f_2$)} hold true. Then, $\mathcal{P}_V$ is a natural cons-traint of problem $(\ref{1.1})$ for $I_V$.
\end{corollary}

\begin{proof}
	Let $u \in \mathcal{P}_V$ be a critical point of the functional $I_V$, constrained to the manifold $\mathcal{P}_V$. Since  $J_V'(u) \neq 0$, it follows from theorem of Lagrange multipliers that there exists $\mu \in \mathbb{R}$ such that $$I_V'(u)+\mu J_V'(u)=0.$$
	Note that the above expression can be associated with the equation
	\begin{equation}\label{4.9}
	-(1+\mu (N-2))\Delta v+(1+\mu N)V(x)v+\mu\nabla V(x) \cdot xv=(1+\mu N)f(v),
	\end{equation}
	by taking $\alpha_n = 0$, $u_n = u$ and $\lambda_n = \mu$ in equation $(\ref{4.5})$. Thus, arguing as in Lemma \ref{lema4.13}, it follows that $\mu=0$. Therefore, $I_V'(u)=0$, which shows that $u$ is a critical point of $I_V$, concluding the proof.
\end{proof}

We also recall the standard result about the splitting of bounded \textit{(PS)} sequences.
For this purpose, first we need a version of Brezis-Lieb lemma \cite{brli} for $\mathcal{D}^{1,2}\!(\mathbb{R}^N)$ found in \cite{Mederski}, Lemma A.1.

\begin{lemma}\label{Med}
	Suppose that $(u_n) \subset \mathcal{D}^{1,2}\!(\mathbb{R}^N)$ is bounded and 
	$u_n(x) \to u_0(x)$ for a.e. $x \in \mathbb{R}^N$. Then
	\begin{equation} \label{MedBL}
	\lim_{n \to \infty} \left( \int_{\mathbb{R}^{N}}\Psi(u_n)\,dx - \int_{\mathbb{R}^{N}}\Psi(u_n- u_0)\,dx\right)
	=\int_{\mathbb{R}^{N}} \Psi(u_0)\,dx
	\end{equation}
	for any function $\Psi : \mathbb{R} \to \mathbb{R}$ of class $C^1$ such that
	$|\Psi^\prime(s)|\leq C |s|^{2*-1}$ for any $s \in \mathbb{R}$ and some constant $C> 0$.
	\end{lemma}
Next lemma presents a new variant of Lions' Lemma in
$\mathcal{D}^{1,2}(\mathbb{R}^N)$, which was proved by Mederski in \cite[Lemma 1.5]{Mederski}.

\begin{lemma}\label{lema3.3}
	Suppose that $(u_n) \subset \mathcal{D}^{1,2}(\mathbb{R}^N)$ is bounded and for some $r > 0$, 
	\begin{equation}\label{3.6}
	\lim_{n \to \infty}\sup_{y \, \in \, \mathbb{R}^N}\int_{B(y,r)}|u_n|^2dx = 0.
	\end{equation}
	Then, $\lim_{n \to \infty}\int_{\mathbb{R}^N}\Psi(u_n)dx = 0,$ for any continuous function $\Psi:\mathbb{R} \to [0,\infty)$ satisfying
	\begin{equation}\label{3.7}
	\lim_{s \to 0}\frac{\Psi(s)}{|s|^{2^*}} = \lim_{|s| \to \infty}\frac{\Psi(s)}{|s|^{2^*}} = 0.
	\end{equation}
\end{lemma}

\begin{proof}
	Let $\varepsilon > 0$ and $2 < p < 2^*$, given arbitrarily, and suppose that $\Psi:\mathbb{R} \to [0,\infty)$ is a continuous function satisfying (\ref{3.7}). Then, we find $\delta, M \in \mathbb{R}$ with $0 < \delta < M$ and $C_\varepsilon > 0$ such that
	\noindent
	\medskip
	\newline
	$(i) \,\Psi(s) \leq \varepsilon|s|^{2^*}$, \ \ for $|s| \leq \delta$;
	\medskip
	\newline
	$(ii)\,\Psi(s) \leq \varepsilon|s|^{2^*}$, \ \ for $|s| > M;$
	\medskip
	\newline
	$(iii)\,\Psi(s) \leq C_\varepsilon|s|^{p}$, \ \ for $|s| \in (\delta,M]$.
	\medskip
	
	Hence, in the view of Lions' lemma we get
	$$\limsup_{n \to \infty}\int_{\mathbb{R}^N}\Psi(u_n)dx \leq \varepsilon\limsup_{n \to \infty}\int_{\mathbb{R}^N}|u_n|^{2}+|u_n|^{2^*}dx.$$ 
	Since $(u_n)$ is bounded in $L^2(\mathbb{R}^N)$ and 
	$L^{2^*}(\mathbb{R}^N)$, we may take the limit $\varepsilon  \to 0$ and  conclude the proof.
\end{proof}

\begin{lemma}\label{lema4.15}
	Let $(u_n)$ be a bounded sequence in $\mathcal{D}^{1,2}\!(\mathbb{R}^N)$ and $d>0$ a constant such that $$I_V(u_n) \to d \;\; \textrm{and} \quad I_V'(u_n) \to 0 \, in \, \left(\mathcal{D}^{1,2}(\mathbb{R}^N)\right)'.$$ Replacing $(u_n)$ by a subsequence, if necessary, there exists a solution $\bar{u}$ of problem (\ref{1.1}), a number $r \in \mathbb{N} \, \cup \, \{0\}$, $r$ nontrivial solutions $w^1,\cdots,w^r$ of the limit problem (\ref{P_0}) and $r$ sequences $(y_n^j) \subset \mathbb{R}^N$, $1 \leq j \leq r$, satisfying: \\
	(i) $|y_{n}^{j}| \to \infty \quad \textrm{and} \quad |y_{n}^{j}-y_{n}^{i}| \to \infty, \ \textrm{if} \ i\neq j$;\\
	(ii) $u_n-\sum_{j=1}^{r}w^j(\cdot - y_n^j) \to \bar{u} \ \ \textrm{in} \ \mathcal{D}^{1,2}\!(\mathbb{R}^N)$;\\
	(iii) $d=I_V(\bar{u})+\sum_{j=1}^{r}I_0(w^j)$,\\
	for $r \in \mathbb{N}$. In the case $r=0$, the above holds without $w^j$, $(y_n^j)$.
\end{lemma}

\begin{proof}
The proof follows closely the proof of Lemma 3.9 in \cite{Clapp-Maia} using Lemmas \ref{Med} and \ref{lema3.3} either for $\Psi(u)= F(u)$ or $\Psi(u)=f(u)u$, $u \in \mathcal{D}^{1,2}\!(\mathbb{R}^N)$, wherever convenient.
\end{proof}

\begin{lemma}\label{lema4.16}
	Assume that {\bf ($V_1$)--($V_4$)} and {\bf ($f_1$)--($f_3$)} hold true. If $p_V$ is not attained for $I_V$ in $\mathcal{P}_V$, then $p_V \geq p_0$ and every bounded (PS)$_d$-sequence in $\mathcal{D}^{1,2}\!(\mathbb{R}^N)$ has a convergent subsequence, at any level $d \in (p_0,2p_0)$.
\end{lemma}

\begin{proof}
	First let us prove that $p_V \geq p_0$. Indeed, let $(u_n) \subset \mathcal{D}^{1,2}(\mathbb{R}^N)$ be a bounded sequence and (PS) at level $p_V$, i.e. $I_V(u_n) \to p_V$ and $I'_V(u_n) \to 0$ in $(\mathcal{D}^{1,2}\!(\mathbb{R}^N))^\prime$. By Lemma \ref{lema4.4}, we have $p_V > 0$, and by using Lemma \ref{lema4.15}, it follows that $p_V \geq p_0$. 
Now, let us prove that every bounded (PS)$_d$-sequence in $\mathcal{D}^{1,2}\!(\mathbb{R}^N)$ has a convergent subsequence, at every level $d \in (p_0,2p_0)$. Indeed, given $d \in (p_0,2p_0)$, we take a bounded sequence $(u_n) \subset \mathcal{D}^{1,2}\left(\mathbb{R}^N\right)$ such that $I_V(u_n) \to d$ and $\Vert I'_V(u_n) \Vert_{\left(\mathcal{D}^{1,2}(\mathbb{R}^N)\right)'} \to 0$. Then, using Lemma \ref{lema4.15}, after passing to a subsequence, we obtain 
	\begin{equation} \label{uconvergence}
	u_n - \sum_{j=1}^{r}w^j(\cdot - y_{n}^{j}) \to \bar{u} \qquad \textrm{in} \ \ \mathcal{D}^{1,2}\!(\mathbb{R}^N),
	\end{equation} 
	where $w^j$ is a weak solution of the problem (\ref{P_0}), $(y_n^j) \subset \mathbb{R}^N$ with $|y_{n}^{j}| \to \infty$ and $\bar{u}$ is a weak solution of the problem (\ref{1.1})
	and by Lemma \ref{lema4.4} it follows that $I_V(\bar{u}) > 0$. Thus, since $d \in (p_0,2p_0)$ from Lemma \ref{lema4.15} $(iii)$, then $r < 2$. If $r = 1$, we have two cases:\\
	(i) $\bar{u} \neq 0$. In this case, we have $I_V(\bar{u}) \geq p_0$. Moreover, $I_0(w^1) = p_0$, then it follows that $d \geq 2p_0$.\\ 
	(ii) $\bar{u} = 0$. In this case, since $w$ is the unique positive solution (up to translations) of the problem (\ref{P_0}), we have $d = I_0(w^1) = I_V(w)  = p_0$.
	In both cases, we get a contradiction, since $d \in (p_0,2p_0)$. Therefore, we must have $r = 0$ and it follows that $u_n \to \bar{u}$ in $\mathcal{D}^{1,2}\!(\mathbb{R}^N)$.
\end{proof}

\begin{corollary}\label{cor4.17}
	Assume that {\bf ($V_1$)--($V_5$)} and {\bf ($f_1$)--($f_3$)} hold true. If $p_V$ is not attained for $I_V$ in $\mathcal{P}_V$, $(u_n)$ is a sequence in $\mathcal{P}_V$ such that $I_V(u_n) \to d$, with $d \in (p_0,2p_0)$, and $I_V|_{\mathcal{P}_V}'(u_n) \to 0$ in $(\mathcal{D}^{1,2}\!(\mathbb{R}^N))^\prime$, then $(u_n)$ is relatively compact in $\mathcal{D}^{1,2}\!(\mathbb{R}^N)$, i.e. after passing to a subsequence, there exists $\bar{u} \in \mathcal{P}_V$ such that $u_n \to \bar{u}$.
\end{corollary}

\begin{proof}
	Using Lemmas \ref{lema4.12} and \ref{lema4.16}, it follows that $p_V=p_0$. By assumption, we have $(u_n) \subset \mathcal{P}_V$ is a sequence such that $I_V(u_n) \to d$ and $I_V|_{\mathcal{P}_V}'(u_n) \to 0$. Then, using the Lemma \ref{lema4.13} we have $I'_V(u_n) \to 0$ in $(\mathcal{D}^{1,2}\!(\mathbb{R}^N))^\prime$ and by Lemma \ref{lema3.2} it follows that $(u_n)$ is bounded in $\mathcal{D}^{1,2}\!(\mathbb{R}^N)$. Thus, by Lemma \ref{lema4.15}, after passing to a subsequence, we get \eqref{uconvergence}
	where $w^j$ is a weak solution of the problem (\ref{P_0}), $(y_n^j) \subset \mathbb{R}^N$ with $|y_{n}^{j}| \to \infty$ and $\bar{u}$ is a weak solution of problem (\ref{1.1}). Therefore, it follows from Lemma \ref{lema4.16} that $u_n \to \bar{u}$, with $\bar{u} \in \mathcal{P}_V$.
\end{proof}

\section{Existence of a positive solution}


We will need the following result of \cite[Lemma 4.1]{Clapp-Maia} and we refer  to that for the proof .
\begin{lemma}\label{ClappM}
	\label{power}
	\begin{enumerate}
		\item[{\rm (a)}] If $y_{0},y\in\mathbb{R}^{N},$ $y_{0}\neq y,$ and $\alpha$ and $\beta$ are positive constants such that $\alpha+\beta>N$, then there exists $C_{1}=C_{1}(\alpha,\beta,\lvert y-y_{0}\rvert)>0$ such that
		\[
		\int_{\mathbb{R}^{N}}\frac{\mathrm{d}x}{(1+|x-Ry_{0}|)^{\alpha}%
			(1+|x-Ry|)^{\beta}}\leq C_{1}R^{-\mu}%
		\]
		for all $R\geq1,$ where $\mu:=\min\{\alpha,\beta,\alpha+\beta-N\}$.
		
		\item[{\rm (b)}] If $y_{0},y\in\mathbb{R}^{N}\setminus\{0\},$ and $\theta$ and $\gamma$ are positive constants such that $\theta+2\gamma>N$, then there exists $C_{2}=C_{2}(\theta,\gamma,\lvert y_{0}\rvert,\lvert y\rvert)>0$ such that
		\[
		\int_{\mathbb{R}^{N}}\frac{\mathrm{d}x}{(1+|x|)^{\theta}(1+|x-Ry_{0}%
			|)^{\gamma}(1+|x-Ry|)^{\gamma}}\leq C_{2}R^{-\tau},
		\]
		for all $R\geq1,$ where $\tau:=\min\{\theta,2\gamma,\theta+2\gamma-N\}$.
	\end{enumerate}
\end{lemma}

In this section we will prove our main result. Its proof requires some important estimates and the previous lemmata.

In what follows, we will consider $y_0 \in \mathbb{R}^N$ a fixed vector, with $|y_0|=1$ and $w$ the positive radial ground state solution of the limit problem (\ref{P_0}). Let $B_r(x_0):=\{x \in \mathbb{R}^N: |x-x_0|\leq r\}$. For any $y \in \partial B_2(y_0)$, $R > 0$ and every $\lambda \in (0,1)$, we write 
\begin{equation}\label{5.1}
	w_{0,\lambda}^R:=w\left(\frac{\cdot-Ry_0}{\lambda}\right), \quad w_{y,1-\lambda}^R:=w\left(\frac{\cdot-Ry}{1-\lambda}\right)
\end{equation}
and, for $\lambda=0$ or $\lambda=1$, we write, respectively, 
\begin{equation}\label{5.2}
	w_{0,\lambda}^R:=0, \quad w_{y,1-\lambda}^R:=0.
\end{equation}

In the next lemmas we study the interaction of powers of these two translated and dilated solitons.
	
	\begin{lemma}\label{w_0w_y}
	Let $\bar{\alpha}$ and $\bar{\beta}$ be constants such that $2\bar{\alpha} > 2^*$ and $\bar{\beta} \geq 1$. Then, for any $R \geq 1$, $y \in \partial B_2(y_0)$ and $\lambda \in [0,1]$, there exist constants $C_3=C_3(N,\bar{\alpha},\bar{\beta})>0$ and $C_4=C_4(N,\bar{\alpha},\bar{\beta})>0$ such that 
	\begin{equation}\label{5.5}
	\int_{\mathbb{R}^N}\left(w_{0,\lambda}^R\right)^{\bar{\alpha}}\left(w_{y,1-\lambda}^R\right)^{\bar{\beta}} \leq C_3R^{-(N-2)},
	\end{equation}
	and
	\begin{equation}\label{5.6}
	\int_{\mathbb{R}^N}\left(w_{y,1-\lambda}^R\right)^{\bar{\alpha}}\left(w_{0,\lambda}^R\right)^{\bar{\beta}} \leq C_4R^{-(N-2)}.
	\end{equation} 
\end{lemma}

\begin{proof}
	If $\lambda = 0$ or $\lambda = 1$, the result follows trivially using the definitions (\ref{5.2}). Suppose now that $\lambda \in (0,1)$ and observe that
		\begin{equation}\label{deslambda}
		1+\left|\frac{x-Ry_0}{\lambda}\right| \geq 1+\left|x-Ry_0\right| \qquad \textrm{and} \qquad 1+\left|\frac{x-Ry}{1-\lambda}\right| \geq 1+\left|x-Ry\right|,
	\end{equation}

	 so by inequalities in $(\ref{3.1})$  there exists $C > 0$ such that
	\begin{eqnarray*}
	 \int_{\mathbb{R}^N}\left(w\left(\frac{x-Ry_0}{\lambda}\right)\!\right)^{\bar{\alpha}}\left(w\left(\frac{x-Ry}{1-\lambda}\right)\!\right)^{\bar{\beta}} 
		&\leq& C\!\int_{\mathbb{R}^N}\left(1+\left|\frac{x-Ry_0}{\lambda}\right|\right)^{-\bar{\alpha}(N-2)}\left(1+\left|\frac{x-Ry}{1-\lambda}\right|\right)^{-\bar{\beta}(N-2)} \\
		&\leq& C\!\int_{\mathbb{R}^N}\left(1+\left|x-Ry_0\right|\right)^{-\bar{\alpha}(N-2)}\left(1+\left|x-Ry\right|\right)^{-\bar{\beta}(N-2)}.
	\end{eqnarray*}
	Since $\bar{\alpha} > 2^*/2$ and $\bar{\beta} \geq 1$, then $\bar{\alpha}(N-2) > N$ and $\bar{\beta}(N-2) \geq N-2$. Therefore, we can apply Lemma \ref{ClappM}(a) with $\alpha = \bar{\alpha}(N-2)$ and $\beta = \bar{\beta}(N-2)$, in which $\mu:=\min\{\alpha,\beta,\alpha+\beta-N\} \geq N-2$, to obtain $C_3 > 0$ such that  
	\begin{eqnarray*}
		\int_{\mathbb{R}^N}\left(w_{0,\lambda}^R\right)^{\bar{\alpha}}\left(w_{y,1-\lambda}^R\right)^{\bar{\beta}} \leq C_3R^{-(N-2)}.
	\end{eqnarray*}
	Similarly, there exists $C_4 > 0$ such that
	\begin{eqnarray*}
		\int_{\mathbb{R}^N}\left(w_{y,1-\lambda}^R\right)^{\bar{\alpha}}\left(w_{0,\lambda}^R\right)^{\bar{\beta}} \leq C_4R^{-(N-2)}.
	\end{eqnarray*}
	
\end{proof}

Now, for every $\lambda \in [0,1]$, we will define
\begin{equation}\label{5.4}
\varepsilon_{\lambda}^R(y):=\int_{\mathbb{R}^N}f\!\left(w_{0,\lambda}^R\right)w_{y,1-\lambda}^R\,dx. 
\end{equation}
We will obtain some estimates for $\varepsilon_{\lambda}^R$ and show they do not depend on $y$.

\begin{lemma}\label{varep_m}
	Assume that {\bf ($f_1$)} holds true. Then, there exists a constant $C>0$ such that 
	\begin{equation}\label{desiepsilon}
	\varepsilon_{\lambda}^R \leq CR^{-(N-2)},
	\end{equation}
	for all $y \in \partial B_2(y_0)$, $\lambda \in [0,1]$ and $R \geq 1$.
\end{lemma}

\begin{proof}
	If $\lambda = 0$ or $\lambda = 1$, the result follows trivially, using the definitions in (\ref{5.2}). Suppose that $0 < \lambda < 1$ and let $\varepsilon_{\lambda}^R$ be defined as in (\ref{5.4}). Using hypothesis {\bf ($f_1$)}, we have
	\begin{eqnarray*}
		\varepsilon_{\lambda}^R &=& \int_{\mathbb{R}^N}f\left(w_{0,\lambda}^R\right)w_{y,1-\lambda}^R \leq A_2\!\int_{\mathbb{R}^N}\left(w_{0,\lambda}^R\right)^{2^*-1}w_{y,1-\lambda}^R.
	\end{eqnarray*}
	Since $2^*-1 > 2^*/2$, applying Lemma \ref{w_0w_y} with $\bar{\alpha}=2^*-1$ and $\bar{\beta}=1$, there exists $C>0$ such that $$\varepsilon_{\lambda}^R \leq CR^{-(N-2)}.$$
	
\end{proof}

		\begin{lemma}\label{lema5.4}
			Assume $(f_1)$, then there exists a constant $C>0$ such that $$\varepsilon_{\lambda}^R \geq C\lambda_{-}^{N}R^{-(N-2)},$$ where $\lambda_{-}:=\min\{\lambda, 1-\lambda\}$, for all  $y \in \partial B_2(y_0)$, $\lambda \in (0,1)$ and $R\geq 1$.
		\end{lemma}
		
		\begin{proof}
			First note that, for every $R \geq 1$, if $z \in B_1(0)$, it holds 
			\begin{eqnarray}\label{desi}
				1+\left|\frac{\lambda z}{1-\lambda}-\frac{R(y-y_0)}{1-\lambda}\right| \!\!\!&=&\!\!\!\! 1+\frac{\lambda}{1-\lambda}\left|z-\frac{R(y-y_0)}{\lambda}\right| \leq\! 1+\frac{\lambda}{1-\lambda}\left(1+\frac{2R}{\lambda}\right) \leq\! \frac{3R}{1-\lambda}.
		\end{eqnarray}
			Furthermore, the estimate $\Vert w\Vert_\infty < \gamma$, for the constant $\gamma$ which appears in $(f_3)$, is given in \cite[Theorem 2]{Tang}. 
			So, there exists a constant $C>0$ such that $f(w(z)) \geq C$, for all $z \in B_1(0)$.
			Thus, a change of variables $z=(x-Ry_0)/{\lambda}$ and  (\ref{3.1}) and (\ref{desi}), yield
			\begin{eqnarray*}
			\varepsilon_{\lambda}^R &=& \displaystyle\int_{\mathbb{R}^N}f\!\left(\!w\!\left(\frac{x-Ry_0}{\lambda}\right)\right)w\!\left(\frac{x-Ry}{1-\lambda}\right) = \lambda^N\!\displaystyle\int_{B_1(0)}f(w(z))\,w\!\left(\frac{\lambda z}{1-\lambda}-\frac{R(y-y_0)}{1-\lambda}\!\right) \\
				&\geq& C\left(\frac{\lambda(1-\lambda)}{3}\right)^{\!N}\left|B_1(0)\right|R^{-(N-2)} \geq C\lambda_-^{\!N} R^{-(N-2)},
			\end{eqnarray*}
			since $\lambda_-/2 \leq \lambda(1-\lambda)$ and the desired inequality follows. 
		\end{proof}
Observe that the lower bound obtained for $\varepsilon_{\lambda}^R$	depends on $\lambda$, while the upper bound is uniform for all $\lambda$ in $[0,1]$. However, in any closed sub-interval	of $(0,1)$ the upper and lower bounds for $\varepsilon_{\lambda}^R$ are independent of $\lambda$. This is going to be crucial in the end.

Anologously, the same upper and lower bounds are obtained for the integral
\[
\int_{\mathbb{R}^N}f\!\left(w_{y,1-\lambda}^R\right)w_{0,\lambda}^R\,dx= O(\varepsilon_{\lambda}^R).
\]	
	
	The next lemma presents the order of interaction between the gradients of two translated solitons.
	\begin{lemma}\label{gradw_0.gradw_y}
		For every $R \geq 1$, $y \in \partial B_2(y_0)$ and $\lambda \in [0,1]$, there exists a constant $C=C(\lambda) > 0$ such that 
		\begin{eqnarray}\label{equivlema57}
			\int_{\mathbb{R}^N}\nabla w_{0,\lambda}^R \cdot \nabla w_{y,1-\lambda}^R\,dx \leq CR^{-(N-2)}.
		\end{eqnarray}
	\end{lemma}
	
	\begin{proof}
		If $\lambda = 0$ or $\lambda = 1$, the result follows trivially, using the definitions (\ref{5.2}). Suppose now that $\lambda \in (0,1)$ and observe that, taking the derivatives and using (\ref{3.2}) and (\ref{deslambda})
		\begin{eqnarray*}
			\int_{\mathbb{R}^N}\nabla w_{0,\lambda}^R \cdot \nabla w_{y,1-\lambda}^R\,dx &=& \frac{1}{\lambda(1-\lambda)}\int_{\mathbb{R}^N}\nabla w\!\left(\frac{x-Ry_0}{\lambda}\right) \cdot \nabla w\!\left(\frac{x-Ry}{1-\lambda}\right)dx \\
		&\leq& \frac{C}{\lambda(1-\lambda)}\!\int_{\mathbb{R}^N}\left(1+\left|x-Ry_0\right|\right)^{-(N-1)}\left(1+\left|x-Ry\right|\right)^{-(N-1)}dx.
		\end{eqnarray*}
	By Lemma $\ref{ClappM}$ $(a)$, with $\alpha=\beta=N-1$, so that $\mu=N-2$, the inequality $(\ref{equivlema57})$ follows and the lemma is proved.
		\end{proof}
	We will need the following estimates adapted from a result in \cite[Lemma 2.2]{Ack-Clapp-Pa}.

	\begin{lemma}\label{A-Clapp-P}
	Assume {\bf ($f_1$)}, then there exists $\sigma \in\left(1/2,1\right]$ with the following property: for any given $C_5\geq 1$ there is a constant $C_6>0$ such that the inequalities $$|f(u+v)-f(u)-f(v)| \leq C_6 |uv|^\sigma$$ and $$|F(u+v)-F(u)-F(v)-f(u)v-f(v)u|\leq C_6|uv|^{2\sigma}$$ hold true for all $u,v \in \mathbb{R}$, with $|u|,|v| \leq C_5$.
\end{lemma}
\begin{proof} Hypothesis ($f_1$) implies there exists a constant $C>0$ such that $\left|f^{(i)}(s)\right| \leq C|s|^{2^*-(i+1)}$, for $i=1, 2, 3$, and $|s| \leq C_5$.
	Set $\sigma:=\min\left\{{2^*}/{4},1\right\}=\min\left\{{N}/{(2(N-2))},1\right\} \in ({1}/{2}, 1]$. The proof of the inequalities follows by simple calculations as in \cite{Ack-Clapp-Pa}.
	\end{proof}
Let us define the sum of the two translated solitons
\begin{equation}\label{5.3}
U_{y,\lambda}^R:=w_{0,\lambda}^R+w_{y,1-\lambda}^R
\end{equation}
and present some of its properties and estimates.

\begin{corollary}\label{F-ACP-estimate}
	Assume {\bf ($f_1$)--($f_2$)}. Then, it holds
	\begin{equation}\label{ACP-estimate}
\int_{\mathbb{R}^N} F(U_{y,\lambda}^R) - F(w_{0,\lambda}^R) - F(w_{y,1-\lambda}^R) - f(w_{0,\lambda}^R)w_{y,1-\lambda}^R - f(w_{y,1-\lambda}^R)w_{0,\lambda}^R \;dx= o(\varepsilon_{\lambda}^R).
	\end{equation}
\end{corollary}
\begin{proof}
		For simplicity, set $w_0:=w_{0,\lambda}^R$, $w_y:=w_{y,1-\lambda}^R$ and $U:=U^R_{y, \lambda}$.
		If $N \geq 4$, then $\sigma=\min\left\{{2^*}/{4},1\right\}= {2^*}/{4}={N}/{(2(N-2))}$ and $\mu= \min\left\{2 \sigma (N-2), 4 \sigma (N-2) -N\right\} > N-2$. Thus,  
			Lemmas \ref{w_0w_y}, \ref{varep_m} and \ref{lema5.4} give that $|w_0 w_y|^{2\sigma} \leq C R^{-\mu}= o(\varepsilon_{\lambda}^R)$.
			
		The case $N=3$ is a little more delicate since $\sigma =1$ and $\mu=1$, which gives $|w_0 w_y|^{2\sigma} \leq C R^{-1}= O(\varepsilon_{\lambda}^R)$. Using hypothesis $(f_1)$ for $i=3$ in the proof of Lemma \ref{A-Clapp-P}, in fact we can obtain $C>0$ such that
		\[
		|F(U) - F(w_0) - F(w_y) - f(w_0)w_y - f(w_y)w_0|
		\leq \; C \left[w_0^4w_y^2 +w_0^3w_y^3 + w_0^2w_y^4\right]
	\leq  \; C R^{-2}= o(\varepsilon_{\lambda}^R),
		\]
		which yields \eqref{ACP-estimate}, and the proof is complete.
	\end{proof}

	\begin{lemma}\label{lema.aux}
		Assume {\bf ($V_1$)--($V_2$)} and {\bf ($f_1$)--($f_2$)}. Then,  the following statements hold: \\
		$(a)$ $\displaystyle\int_{\mathbb{R}^N}|\nabla U^R_{y,\lambda}|^2dx = C_\lambda^{N-2}\!\!\int_{\mathbb{R}^N}|\nabla w|^2dx + o_R(1)$;\\
	$(b)$ $\displaystyle\int_{\mathbb{R}^N}F(U^R_{y, \lambda})dx = C_\lambda^{N}\!\!\int_{\mathbb{R}^N}F(w)dx + o_R(1)$,\\
			%
		where $C_\lambda^j:=\lambda^{j}+(1-\lambda)^{j}$ and $o_R(1) \to 0$ as $R \to +\infty$, uniformly for all $y \in \partial B_2(y_0)$ and $\lambda \in [0,1]$.
	\end{lemma}
	
	\begin{proof}
		For simplicity, set $w_0:=w_{0,\lambda}^R$, $w_y:=w_{y,1-\lambda}^R$ and $U:=U^R_{y, \lambda}$. If $\lambda = 0$ or $\lambda = 1$, the statements follow trivially for all $y \in \partial B_2(y_0)$ and $o_R(1)=0$, using (\ref{5.2}) and (\ref{5.3}). Suppose now that $0 < \lambda < 1$, then we have
		\begin{eqnarray*}
			\int_{\mathbb{R}^N}|\nabla U|^2dx
			&=& \lambda^{N-2}\!\!\int_{\mathbb{R}^N}|\nabla w|^2dx + (1-\lambda)^{N-2}\!\!\int_{\mathbb{R}^N}|\nabla w|^2dx + 2\!\int_{\mathbb{R}^N}\nabla w_0 \cdot \nabla w_ydx \\
			&=& C_\lambda^{N-2}\!\!\int_{\mathbb{R}^N}|\nabla w|^2dx + 2\!\int_{\mathbb{R}^N}\nabla w_0 \cdot \nabla w_ydx.
		\end{eqnarray*}
	 By Lemma \ref{gradw_0.gradw_y}, there exists  $C > 0$ such that
		\begin{eqnarray*}
			\int_{\mathbb{R}^N}\nabla w_0 \cdot \nabla w_y\,dz \leq CR^{-(N-2)},
		\end{eqnarray*}
	 proving item { (a)}. We also have
		\begin{align*}
		& \int_{\mathbb{R}^N}F(U)dx - C_\lambda^{N}\!\!\int_{\mathbb{R}^N}F(w)dx = \int_{\mathbb{R}^N}F(U)dx - \lambda^N\!\!\int_{\mathbb{R}^N}F(w)dx - (1-\lambda)^N\!\!\int_{\mathbb{R}^N}F(w)dx \\
	& \qquad\qquad\qquad = \int_{\mathbb{R}^N}[F(U) - F(w_0) - F(w_y) - f(w_0)w_y - f(w_y)w_0]dx \\
		& \qquad\qquad\qquad \ \ \, + \int_{\mathbb{R}^N}[f(w_0)w_y + f(w_y)w_0]dx.
		\end{align*}
		From \eqref {5.4}, \eqref{desiepsilon} and \eqref{ACP-estimate}, there exists $C > 0$ such that $$\int_{\mathbb{R}^N}\left|F(U) - F(w_0) - F(w_y) - f(w_0)w_y - f(w_y)w_0\right|dx \leq CR^{-(N-2)},$$ $$\int_{\mathbb{R}^N}[f(w_0)w_y + f(w_y)w_0]dx = 2\varepsilon_{\lambda}^{R} \leq CR^{-(N-2)},$$ for every $y \in \partial B_2(y_0)$, $\lambda \in (0,1)$ and $R \geq 1$ so {(b)} follows, concluding the proof of the lemma.
		
	\end{proof}

	\begin{lemma}\label{unic.s}
		Assume that {\bf ($V_1$)--($V_3$)} and {\bf ($f_1$)--($f_2$)} hold true. Then, there exists $R_0>1$ such that given $y \in \partial B_2(y_0)$, $\lambda \in [0,1]$ and $R \geq R_0$, there exists a unique positive constant $s:=S_{y,\lambda}^{R}$ such that $$U_{y,\lambda}^R\!\left(\frac{\cdot}{s}\right) \in \mathcal{P}_V.$$  Moreover, there exist $\sigma_0 \in (0, 1)$ and $S_0 >1$ such that $S_{y,\lambda}^{R} \in (\sigma_0,S_0)$ for any $y \in \partial B_2(y_0)$, $\lambda \in [0,1]$ and $R \geq R_0$. In addition, $S_{y,\lambda}^{R}$ is a continuous function of the variables $y$, $\lambda$ and $R$.
	\end{lemma}
	
	\begin{proof}
Denote, as before, $U:=U_{y,\lambda}^R$ and let $\xi_V:(0,+\infty) \to \mathbb{R}$ be defined by
		\begin{eqnarray*}
		\xi_V(s) := I_V\!\left(U(\cdot/s)\right) = \frac{s^{N-2}}{2}\!\!\int_{\mathbb{R}^N}|\nabla U|^2dx + \frac{s^{N}}{2}\!\!\int_{\mathbb{R}^N}V(sx)U^2dx - s^N\!\!\int_{\mathbb{R}^N}F(U)dx.
		\end{eqnarray*}
		Then, $U(\cdot/s) \in \mathcal{P}_V$ if and only if $\xi_V'(s)=0$, where 
		\begin{eqnarray*}
			\xi_V'(s) = s^{N-3}\left[\frac{N-2}{2}\!\!\int_{\mathbb{R}^N}|\nabla U|^2dx - Ns^2\!\!\int_{\mathbb{R}^N}F(U)dx
			 + \frac{N}{2}s^2\!\!\int_{\mathbb{R}^N}\left(\frac{\nabla V(sx) \cdot (sx)}{N}+V(sx)\!\right)\!U^2dx\right].
		\end{eqnarray*}
		Since $s>0$, we have $\xi_V'(s)=0$ if and only if $$\frac{N-2}{2}\!\!\int_{\mathbb{R}^N}|\nabla U|^2dx = Ns^2\left[\int_{\mathbb{R}^N}F(U)dx - \frac{1}{2}\int_{\mathbb{R}^N}\left(\frac{\nabla V(sx) \cdot (sx)}{N}+V(sx)\!\right)\!U^2dx\right].$$ Set as before $C_\lambda^j:=\lambda^{j}+(1-\lambda)^{j}$ with $j \in \mathbb{N}$ and note that $2^{-j} \leq C_\lambda^j \leq 2$, for every $j \in \mathbb{N}$ and $\lambda \in [0,1]$. Moreover, observe that
		\begin{eqnarray*}
			\int_{\mathbb{R}^N}U^2dx &=& \int_{\mathbb{R}^N}\left(w_0+w_y\right)^2dx 
			=C_\lambda^{N}\!\!\int_{\mathbb{R}^N}w^2 dx + o_R(1)
		\end{eqnarray*}
	which gives that $\|U\|_{2}$ is bounded uniformly for $y\in \partial B_2(y_0)$, $\lambda\in (0,1)$ and $R\geq 1$.\\
		Since $\int_{\mathbb{R}^N}|\nabla w|^2dx > 0$, using {\bf ($V_2$)} and Lemma \ref{lema.aux}, there exists $R_1>1$, sufficiently large, and  $\sigma_0 \in (0,1)$ sufficiently small such that 
		\begin{eqnarray*}
			\xi_V'(s)=\! s^{N-3}\!\left\{\frac{N-2}{2}\!\!\int_{\mathbb{R}^N}\!|\nabla U|^2dx - Ns^2\left[\int_{\mathbb{R}^N}\!F(U)dx - \frac{1}{2}\!\!\int_{\mathbb{R}^N}\!\left(\frac{\nabla V(sx) \cdot (sx)}{N}+V(sx)\!\right)\!U^2dx\right]\right\}\! >\! 0,
		\end{eqnarray*}
		for every $s \in (0,\sigma_0]$, $y \in \partial B_2(y_0)$, $\lambda \in [0,1]$ and $R \geq R_1$.
		
		Now let us define a function $\psi_V:(\sigma_0,+\infty) \to \mathbb{R}$ by
		\begin{eqnarray*}
			\psi_V(s) = s^2\!\left[\int_{\mathbb{R}^N}F(U)dx - \frac{1}{2}\int_{\mathbb{R}^N}\left(\frac{\nabla V(sx) \cdot (sx)}{N}+V(sx)\!\right)\!U^2dx\right].
		\end{eqnarray*}
		Note that
		\begin{eqnarray*}
			\psi_V'(s) 
					&=& 2s\left[\int_{\mathbb{R}^N}F(U)dx - \frac{1}{2}\int_{\mathbb{R}^N}V(sx)U^2dx\right] \\
			&& - \, \frac{s}{2}\left[(N+3)\!\int_{\mathbb{R}^N}\frac{\nabla V(sx) \cdot (sx)}{N}U^2dx + \int_{\mathbb{R}^N}\frac{(sx) H(sx) (sx)}{N}U^2dx\right].
		\end{eqnarray*}
		Observe that
		$$(1+|sx|)^{-k} \leq \left\{\begin{array}{cc}
		\sigma_0^{-k}(1+|x|)^{-k}, & \ \ \text{if} \ \ \sigma_0 < s \leq 1 \\ 
		(1+|x|)^{-k}, & \text{if} \ \ 1 \leq s.
		\end{array}\right.$$
		Therefore, using the hypothesis {\bf ($V_2$)}, we obtain constants $C>0$ such that
		\begin{eqnarray*}
			\int_{\mathbb{R}^N}|V(sx)|U^2dx &\leq& C\!\int_{\mathbb{R}^N}(1+|x|)^{-k}\left(w_0+w_y\right)^2dx,
		\end{eqnarray*}
		\begin{eqnarray*}
			\int_{\mathbb{R}^N}|\nabla V(sx) \cdot (sx)|U^2dx &\leq& C\!\int_{\mathbb{R}^N}(1+|x|)^{-k}\left(w_0+w_y\right)^2dx,
		\end{eqnarray*}
		for every $s > \sigma_0$. Thus, using the inequalities in (\ref{3.1}) and applying  Lemma \ref{ClappM}(b), it follows
		\begin{eqnarray}
			\int_{\mathbb{R}^N}|V(sx)|U^2dx = o_R(1), \qquad \int_{\mathbb{R}^N}|\nabla V(sx) \cdot (sx)|U^2dx = o_R(1),
		\end{eqnarray}
		where $o_R(1) \to 0$ as $R \to +\infty$. Furthermore note that
		\begin{eqnarray*}\label{igualU}
			\int_{\mathbb{R}^N}\left|(sx)H(sx)(sx)\right|U^2dx
			&\leq& 2\!\int_{\mathbb{R}^N}\left|(sx)H(sx)(sx)\right|\left[(w_0)^2 + (w_y)^2\right]dx.
		\end{eqnarray*}
		Let us prove that $\int_{\mathbb{R}^N}\left|(sx)H(sx)(sx)\right|(w_0)^2dx = o_R(1)$. Indeed, if $\lambda = 0$ the result follows from $(\ref{5.2})$. Suppose that $\lambda \in (0,1]$ and let $\varepsilon > 0$ be given arbitrarily. Then, since $\Vert w \Vert_2 > 0$, using the hypothesis {\bf ($V_5$)}, we can take $\tilde\rho > 0$ sufficiently large  such that for all $s>\sigma_0$  and $|x|\geq \tilde\rho/\sigma_0$,
		$$\left|(sx)H(sx)(sx)\right| < \frac{\varepsilon}{4\Vert w \Vert_2^2}.$$ 
		So, for all $s > \sigma_0$ and $\lambda \in (0,1]$, we have
		\begin{equation}\label{5.8}
		\int_{|x|\geq \tilde\rho/\sigma_0}\left|(sx)H(sx)(sx)\right|(w_0)^2dx \leq \frac{\varepsilon}{4\Vert w \Vert_2^2}\int_{\mathbb{R}^N}w_0^2\,dx \leq \frac{\varepsilon}{4\Vert w \Vert_2^2}\lambda^N\!\int_{\mathbb{R}^N}w^2dx \leq \frac{\varepsilon}{4}.
		\end{equation}
		On the other hand, using (\ref{1.4}) and (\ref{3.1}), we obtain 
	\begin{align}\label{5.9}
		& \int_{|x|\leq\tilde\rho/\sigma_0}\left|(sx)H(sx)(sx)\right|(w_0)^2dx \leq C\!\int_{|x|\leq\tilde\rho/\sigma_0}(w_0)^2dx \leq C\!\int_{|x|\leq \tilde\rho/\sigma_0}\left(1+\left|\frac{x-Ry_0}{\lambda}\right|\right)^{\!-(N-2)}dx\nonumber \\
		& \qquad\qquad\qquad \leq C\!\int_{|x|\leq\tilde\rho/\sigma_0}(|Ry_0|-|x|)^{-(N-2)}dx \leq C\left(R-\frac{\tilde\rho}{\sigma_0}\right)^{\!-(N-2)}\nonumber \\
		& \qquad\qquad\qquad \leq C\left(\frac{R}{2}\right)^{\!-(N-2)} \leq CR^{-(N-2)},
		\end{align}
		for every $s > \sigma_0$, $\lambda \in (0,1]$. Therefore, inequalities (\ref{5.8}) and (\ref{5.9}) give that
\begin{equation}\label{5.10}
			\int_{\mathbb{R}^N}\left|(sx)H(sx)(sx)\right|(w_0)^2dx \leq \frac{\varepsilon}{4} + CR^{-(N-2)}.
		\end{equation}
		 Since $1 \leq |y| \leq 3$, by an analogous procedure, there exists $C>0$ such that
		\begin{equation}\label{5.11}
			\int_{\mathbb{R}^N}\left|(sx)H(sx)(sx)\right|(w_y)^2dx \leq \frac{\varepsilon}{4} + CR^{-(N-2)},
		\end{equation}
		for every $s > \sigma_0$, $y \in \partial B_2(y_0)$, $\lambda \in [0,1]$. From (\ref{5.10}) and (\ref{5.11}), there exists $C > 0$ such that
\begin{eqnarray}\label{5.12}
			\int_{\mathbb{R}^N}\left|(sx)H(sx)(sx)\right|U^2dx 
			&\leq& 2\!\int_{\mathbb{R}^N}\left|(sx)H(sx)(sx)\right|\left[(w_0)^2 + (w_y)^2\right]dx\nonumber \\
			&\leq& \varepsilon + CR^{-(N-2)}.
		\end{eqnarray}
		 Since $\varepsilon > 0$ was taken arbitrarily, it follows from (\ref{5.12}) that
		\begin{eqnarray}\label{sxh}
			\int_{\mathbb{R}^N}\left|(sx)H(sx)(sx)\right|U^2dx = o_R(1).
		\end{eqnarray}
		Thus, knowing that $\int_{\mathbb{R}^N}F(w)dx > 0$, using the hypotheses {\bf ($V_2$)}, {\bf ($V_5$)}, Lemma \ref{lema.aux} $(b)$, (\ref{igualU}) and $(\ref{sxh})$  we obtain
		\begin{eqnarray*}
			\psi_V'(s) &=& 2s\left[\int_{\mathbb{R}^N}F(U)dx - \frac{1}{2}\int_{\mathbb{R}^N}V(sx)U^2dx\right] \\
			&& - \, \frac{s}{2}\left[(N+3)\!\int_{\mathbb{R}^N}\frac{\nabla V(sx) \cdot (sx)}{N}U^2dx + \int_{\mathbb{R}^N}\frac{(sx) H(sx) (sx)}{N}U^2dx\right] > 0,
		\end{eqnarray*}
		for every $s > \sigma_0$, $y \in \partial B_2(y_0)$, $\lambda \in [0,1]$ and $R \geq R_1$ sufficiently large. This means that $\psi_V(s)$ is increasing for $s>\sigma_0$ and $R$ taken sufficiently large. This implies that the term in the brackets for $\xi'_V(s)$ is decreasing for $s>\sigma_0$, and goes to $-\infty$. Therefore, there is a unique $s=S_{y,\lambda}^{R}>\sigma_0$ such that $\xi_V'(s)=0$, i.e.  $U_{y,\lambda}^R\!\left(\cdot/s\right) \in \mathcal{P}_V$.
		Furthermore, again by Lemma \ref{lema.aux} $(b)$, \eqref{1.2} and \eqref{1.3}
		there exists $R_2 >1 $, sufficiently large, and $S_0 >1$ such that $\xi'_V(s) <0$,
		for all $s > S_0$, $R> R_2$, $y \in \partial B_2(y_0)$ and $\lambda \in [0,1]$.
		Taking $R_0 = \max \{R_1, R_2\}$ the result follows.
		Finally, from the uniform estimates for $U$, $\nabla U$ and $F(U)$ with respect to $y$, $\lambda$ and $R >R_0$, the continuity of $S_{y,\lambda}^{R}$ in these variables is clear, and the proof is complete.
		
	\end{proof}

	From here on, consider $S^R_{y, \lambda}$ as obtained in Lemma \ref{unic.s}.

	\begin{lemma}\label{lambda=1/2}
		Assume {\bf ($V_1$)--($V_3$)} and {\bf ($f_1$)--($f_2$)} hold true. Then, for $\lambda=1/2$, we have $S_{y,1/2}^{R} \to 2$ as $R \to +\infty$ uniformly for $y \in \partial B_2(0)$.
	\end{lemma}
	
	\begin{proof}
		By Lemma \ref{unic.s}, there exist $R_0 \geq 1$ and $S_0 > 2$ such that $S_{y,\lambda}^{R} \in (0,S_0)$ for any $R \geq R_0$, $y \in \partial B_2(y_0)$ and $\lambda \in (0,1)$.  Denoting $\overline{w}_0:=w_{0,1/2}^{R}\left(\frac{\cdot}{2}\right)=w(\cdot - 2Ry_0)$ and $\overline{w}_y:=w_{y,1/2}^{R}\left(\frac{\cdot}{2}\right)=w(\cdot - 2Ry)$, we have 
		\begin{eqnarray*}
			J_0\left(\overline{w}_0+\overline{w}_y\right) 
			&=& \left[\frac{N-2}{2}\!\int_{\mathbb{R}^N}|\nabla w|^2 - N\!\int_{\mathbb{R}^N}F(w)\right] + \left[\frac{N-2}{2}\!\int_{\mathbb{R}^N}|\nabla w|^2 - N\!\int_{\mathbb{R}^N}F(w)\right] \\
			&& + \, (N-2)\!\int_{\mathbb{R}^N}\nabla \overline{w}_{0} \cdot \nabla \overline{w}_{y} - N\!\int_{\mathbb{R}^N}\left[F\left(\overline{w}_{0}+\overline{w}_{y}\right) - F\left(\overline{w}_{0}\right) - F\left(\overline{w}_{y}\right)\right].
			\end{eqnarray*}
			Since $J_0(w)=0$, it follows that	
		\begin{equation}\label{Jsoma}	
			J_0\left(\overline{w}_0+\overline{w}_y\right)= (N-2)\!\int_{\mathbb{R}^N}\nabla \overline{w}_{0} \cdot \nabla \overline{w}_{y} - N\!\int_{\mathbb{R}^N}\left[F\left(\overline{w}_{0}+\overline{w}_{y}\right) - F\left(\overline{w}_{0}\right) - F\left(\overline{w}_{y}\right)\right].
			\end{equation}
		Observe that Lemma \ref{gradw_0.gradw_y} with $\lambda=1/2$ yields		
		\begin{equation}\label{nablanabla}
			\int_{\mathbb{R}^N}\nabla \overline{w}_{0} \cdot \nabla \overline{w}_{y}\,dx  \leq 4 CR^{-(N-2)}.
		\end{equation}
		On the other hand, using \eqref{ACP-estimate} and \eqref{desiepsilon}, we get
		\begin{align*}
		& \left\vert F\!\left(\overline{w}_{0}+\overline{w}_{y}\right) - F\!\left(\overline{w}_{0}\right) - F\!\left(\overline{w}_{y}\right)\right\vert \\ 
		& \qquad \leq \left\vert F\!\left(\overline{w}_{0}+\overline{w}_{y}\right) - F\!\left(\overline{w}_{0}\right) - F\!\left(\overline{w}_{y}\right) - f\!\left(\overline{w}_{0}\right)\overline{w}_{y} - f\!\left(\overline{w}_{y}\right)\overline{w}_{0}\right\vert \\
		& \qquad \ \ + \, \left\vert f\!\left(\overline{w}_{0}\right)\overline{w}_{y} + f\!\left(\overline{w}_{y}\right)\overline{w}_{0}\right\vert \\
		& \qquad \leq CR^{-(N-2)}.
		\end{align*}%
		Therefore, there exists $C> 0$ such that
		\begin{eqnarray}\label{JFR}
			\left|J_0\left(\overline{w}_0+\overline{w}_y\right)\right| \leq CR^{-(N-2)}.
		\end{eqnarray}
		Thus, $J_0\left(\overline{w}_0+\overline{w}_y\right) \to 0$ as $R \to \infty$, uniformly for $y \in \partial B_2(0)$. Then, in the case $\lambda=1/2$, using hypothesis {\bf ($V_2$)}, we obtain
		\begin{eqnarray}\label{jbara}
			J_V\!\left(\!U_{y,1/2}^R \left(\frac{\cdot}{2} \right)\!\right){}\nonumber &=& J_V\left(\overline{w}_0+\overline{w}_y\right) = J_0\left(\overline{w}_{0}+\overline{w}_{y}\right) + \frac{N}{2}\!\int_{\mathbb{R}^N}\left(\frac{\nabla V(x)\cdot x}{N}+V(x)\!\right)\!\left(\overline{w}_{0}+\overline{w}_{y}\right)^{2}\!dx \\
		&\leq& J_0\left(\overline{w}_{0}+\overline{w}_{y}\right) + C\int_{\mathbb{R}^N}(1+|x|)^{-k}\left(\overline{w}_{0}+\overline{w}_{y}\right)^{2}\!dx,
		\end{eqnarray}
		and again using (\ref{3.1}) and Lemma \ref{ClappM}(b) the last integral above is bounded by $ CR^{-(N-2)}$.
	
		From $(\ref{JFR})$ and $(\ref{jbara})$, we get $$\left| J_V\!\left(U_{y,1/2}^R\left(\frac{\cdot}{2}\right)\!\right)\right| \leq CR^{-(N-2)}.$$
		 Therefore, $J_V\!\left(U_{y,1/2}^R\left(\frac{\cdot}{2}\right)\!\right)=o_R(1)$, where $o_R(1) \to 0$ as $R \to \infty$, uniformly for $y \in \partial B_2(0)$. This proves the lemma.
		
	\end{proof}
	
\begin{lemma}\label{lema11} Assume that $(V_2)$ holds true. Let  $S_0>2$ and $1\geq \sigma_0>0$, then, there exists $\tau> N-2$ such that the following hold:\\
		$(a)$ $\displaystyle \int_{\mathbb{R}^N}|V(x)|\left(U^R_{y, \lambda}\left(\frac{x}{s}  \right) \right)^2 \,dx\leq CR^{-\tau}$;\\
		$(b)$ $\displaystyle \int_{\mathbb{R}^N}|\nabla V(x)|\left(U^R_{y, \lambda}\left(\frac{x}{s}  \right) \right)^2 \,dx\leq CR^{-\tau}$,\\
		for every $s\in (\sigma_0, S_0)$, $y\in \partial B_2(y_0)$, $\lambda \in [0,1]$ and $R\geq 1$.		
		\end{lemma}
	\begin{proof}
		By $(V_2)$, the decay estimates $(\ref{3.1})$ and inequalities $(\ref{deslambda})$, we obtain
		\begin{eqnarray}\label{desiu}
		\int_{\mathbb{R}^N}|V(x)&\hspace{-.7cm}|&\hspace{-.7cm}\left (U^R_{y, \lambda}\left(\frac{x}{s} \right) \right)^2{}\nonumber\leq 2		\int_{\mathbb{R}^N}|V(x)|\left (w^R_{0, \lambda}\left(\frac{x}{s} \right) \right)^2+ 2	\int_{\mathbb{R}^N}|V(x)|\left (w^R_{y,1- \lambda}\left(\frac{x}{s} \right) \right)^2	\\
		{}\nonumber&\leq& 2s^N	\int_{\mathbb{R}^N}|V(sx)|\left [w^2\left(\frac{x-Ry_0}{\lambda}\right)+ w^2\left(\frac{x-Ry}{1-\lambda}   \right) \right]\\
	&\leq& Cs^N_0\left\{\int_{\mathbb{R}^N}\frac{1}{(1+|sx|)^k(1+|x-Ry_0|)^{2(N-2)}}+ \int_{\mathbb{R}^N}\frac{1}{(1+|sx|)^k(1+|x-Ry|)^{2(N-2)}}\right\}.
		\end{eqnarray}
		Since $1\geq \sigma_0>0$ and $|sx|> \sigma_0|x|$, then by Lemma $\ref{ClappM}$ $(b)$
		$$\int_{\mathbb{R}^N}\frac{1}{(1+|\sigma_0|x||)^k(1+|x-Ry_0|)^{2(N-2)}}\leq \sigma_0^{-k}\int_{\mathbb{R}^N}\frac{1}{(1+|x|)^k(1+|x-Ry_0|)^{2(N-2)}} \leq CR^{-\tau},$$
		where $\tau=\min\,\{k,2(N-2), k+2(N-2)-N\}> N-2$, for every $s\in(\sigma_0,S_0)$, $y\in \partial B_2(y_0)$, $\lambda\in (0,1)$ and $R \geq 1$; analogously for the second integral in $(\ref{desiu})$. Thus, the first inequality of the lemma is proved.\\
		The second assertion of this lemma is obtained is the same way, using $(V_2)$ with $|\nabla V(x).x|\leq A_1(1+|x|)^{-k}$.
		\end{proof}

	\begin{lemma}\label{lema.extrem}
		Assume that {\bf ($V_1$)--($V_3$)} and {\bf ($f_1$)--($f_2$)} hold true. Then, for any $\delta>0$, there exists $R_3 > 0$ such that $$I_V\!\left(\!U_{y,\lambda}^R\!\left(\frac{\cdot}{s}\right)\!\right) < p_0+\delta,$$ for $\lambda = 0$ or $\lambda = 1$ and every $y \in \partial B_2(y_0)$ and $R \geq R_3$, where $s:=S_{y,\lambda}^{R}>0$ is such that $U_{y,\lambda}^R\!\left(\frac{\cdot}{s}\right) \in \mathcal{P}_V$.
	\end{lemma}
	
	\begin{proof}
		Let $\delta > 0$ be given arbitrarily. By Lemma \ref{unic.s}, $S_{y,\lambda}^{R}$ is bounded uniformly in $R$, $y$ and $\lambda$. For $\lambda = 0$, we have $U_{y,\lambda}^R = U_{y,0}^R = w_{y,1}^{R} = w(\cdot - Ry).$ Observe that $w \in \mathcal{P}_0$ and the map $t \mapsto I_0\!\left(\!w\!\left(\frac{\cdot}{t}\right)\!\right)$ is strictly increasing in $(0,1]$ and strictly decreasing in $[1,\infty)$. In particular,
	$		p_0=I_0(w)=\max_{t>0}I_0\!\left(\!w\!\left(\frac{\cdot}{t}\right)\!\right).$ So by changing the variables $x=sz$ and using {\bf ($V_2$)} and (\ref{3.1}), it follows
		\begin{eqnarray*}
			 I_V\!\left(\!U_{y,0}^R\!\left(\frac{\cdot}{s}\right)\!\right) &=& s^{N-2}\left[\frac{1}{2}\int_{\mathbb{R}^N}|\nabla w|^2dx-s^2\int_{\mathbb{R}^N}F(w) dx \right] +\frac{s^N}{2}\int_{\mathbb{R}^N}V(sz)\left( w^R_{y,1}\right)^2dz\\
			 &\leq& I_0(w) + \frac{s^N}{2}\!\int_{\mathbb{R}^N}|V(sz)|(w(z-Ry))^2dz\leq p_0+CR^{-\tau},		 
		\end{eqnarray*}
	by Lemma \ref{lema11},
	where $\tau:=\min\{k,2(N-2),k+2(N-2)-N\} > N-2$. So, given $\delta>0$, there exists $R_3>0$ such that for all $R> R_3$
		\begin{eqnarray*}
			I_V\!\left(\!U_{y,0}^R\!\left(\frac{\cdot}{s}\right)\!\right) &\leq& p_0 + CR^{-\tau} \leq p_0+ CR_1^{-1} < p_0 + \delta,
		\end{eqnarray*}
		for any $y \in \partial B_2(y_0)$. Analogously $$I_V\!\left(\!U_{y,1}^R\!\left(\frac{\cdot}{s}\right)\!\right) < p_0 + \delta,$$ for any $y \in \partial B_2(y_0)$ and $R \geq R_3$.
	\end{proof}

\begin{proposition}\label{prop.princ}
	Assume that {\bf ($V_1$)--($V_3$)} and {\bf ($f_1$)--($f_2$)} hold true. Then, there exist $R_4 \geq 1$ and, for each $R \geq R_4$, a number $\eta=\eta(R) > 0$ such that $$I_V\!\left(\!U_{y,\lambda}^R\!\left(\frac{\cdot}{s}\right)\!\right) \leq 2p_0-\eta,$$ if $s:=S_{y,\lambda}^{R}$, for all  $y \in \partial B_2(y_0)$ and $\lambda \in [0,1]$.
\end{proposition}

\begin{proof}
	If $\lambda = 0$ or $\lambda = 1$, it follows by Lemma \ref{lema.extrem} that, for all $\delta > 0$, there exists $R_1 \geq 1$ such that $$I_V\!\left(\!U_{y,\lambda}^R\!\left(\frac{\cdot}{s}\right)\!\right) < p_0 + \delta,$$ for all $y \in \partial B_2(y_0)$ and $R \geq R_3$. Suppose that $\lambda \in (0,1)$. By Lemma \ref{unic.s}, there exist $R_0 > 0$ and $S_0 > 2$ such that $S_{y,\lambda}^{R} \in (\sigma_0,S_0)$ for all $y \in \partial B_2(y_0)$, $\lambda \in [0,1]$ and $R \geq R_0$,
	 changing the variables $sz=x$ and, for simplicity, denoting $w_0:=w_{0,\lambda}^{R}$ and $w_y:=w_{y,1-\lambda}^{R}$, we have
		\begin{eqnarray*}
		I_V\!\left(\!U_{y,\lambda}^R\!\left(\frac{\cdot}{s}\!\right)\right) 
		&=& \frac{s^{N-2}}{2}\!\left[\int_{\mathbb{R}^N}|\nabla w_0|^2dz-2s^2\!\!\int_{\mathbb{R}^N}F(w_0)dz+\int_{\mathbb{R}^N}|\nabla w_y|^2dz-2s^2\!\!\int_{\mathbb{R}^N}F(w_y)dz\right] \\
		&& + \, \frac{s^{N}}{2}\!\int_{\mathbb{R}^N}V(sz)\!\left[w_0+w_y\right]^2\!dz + s^{N-2}\!\int_{\mathbb{R}^N}\nabla w_0 \cdot \nabla w_y\,dz \\
		&& - \,s^N\!\!\int_{\mathbb{R}^N}\left[F(w_0+w_y)-F(w_0)-F(w_y)-f(w_0)w_y-f(w_y)w_0\right]dz \\
		&& - \,s^N\!\int_{\mathbb{R}^N}\left[f(w_0)w_y + f(w_y)w_0\right]dz \\
(I)	&\leq&
		 I_0(w(\frac{\cdot}{\lambda s})) + I_0(w(\frac{\cdot}{(1-\lambda)s})) \\
	(II)	&& + \, \frac{s^{N}}{2}\!\int_{\mathbb{R}^N}|V(sz)|\!\left[w_0+w_y\right]^2\!dz \\
	(III)	&& - \,s^N\!\!\int_{\mathbb{R}^N}\left[F(w_0+w_y)-F(w_0)-F(w_y)-f(w_0)w_y-f(w_y)w_0\right]dz \\
	(IV)	&& + \,s^{N-2}\!\!\int_{\mathbb{R}^N}\left[\nabla w_0 \cdot \nabla w_y - s^2f(w_0)w_y - s^2f(w_y)w_0\right]\!dz.  
	\end{eqnarray*}
	
	Since $p_0=I_0(w)=\max_{t>0}I_0(w(\frac{\cdot}{t}))$, then
	\begin{eqnarray*}
		(I) 
		&\leq& I_0(w) + I_0(w) = 2p_0.
	\end{eqnarray*}
	By Lemma \ref{lema11} (a), we obtain 
	\begin{eqnarray*}
		(II) \leq CR^{-\tau},
	\end{eqnarray*}
	where $\tau > N-2$ and hence, $(II) \leq o(\varepsilon_{\lambda}^{R})$ for all $N \geq 3$.
	
	Moreover, corollary \ref{F-ACP-estimate}  and $s \leq S_0$ yield
	\[
	(III) = o(\varepsilon_{\lambda}^{R}) 
	\] 
	for all $N \geq 3$.
	
	Now observe that for $\lambda= 1/2$ fixed, using that  $w$ is a solution of (\ref{P_0}), we obtain 
	$$\int_{\mathbb{R}^N}\nabla w^R_{0,1/2}\nabla w^R_{y,1/2}=4\int_{\mathbb{R}^N}f(w^R_{0, 1/2})w^R_{y, 1/2}.$$
	
	By Lemma \ref{lambda=1/2}, we have
	\begin{eqnarray*}
		\lim_{(\lambda, R) \to (1/2, +\infty)}
		s^2\!\int_{\mathbb{R}^N}\left[\frac{f(w_0)w_y+f(w_y)w_0}{2}\right]\!dz
		 &=&
		 4\!\int_{\mathbb{R}^N}\left[\frac{f(w_{0,1/2})w_{y,1/2}+f(w_{y,1/2})w_{0,1/2}}{2}\right]\!dz \\
		 &=& \int_{\mathbb{R}^N}\nabla w_{0,1/2} \cdot \nabla w_{y,1/2}\,dz. 
	\end{eqnarray*}
	Then, taking $R_5$ sufficiently large and $\delta \in (0,1/4)$ sufficiently small, we obtain
	\begin{eqnarray*}
		\frac{4s^2}{3}\!\int_{\mathbb{R}^N}\left[\frac{f(w_0)w_y+f(w_y)w_0}{2}\right]dz \geq 
	\left|\int_{\mathbb{R}^N}\nabla w_0 \cdot \nabla w_y\,dz \right|, 
	\end{eqnarray*}
	for all $\lambda \in [1/2-\delta,1/2+\delta]$, $y \in \partial B_2(y_0)$ and $R \geq R_5$. Thus, there exists $C_0 > 0$ such that
	\begin{eqnarray*}
		(IV) &=& s^{N-2}\!\int_{\mathbb{R}^N}\left[\nabla w_0 \cdot \nabla w_y - s^2f(w_0)w_y - s^2f(w_y)w_0\right]\!dz \\
		&\leq& -\,\frac{s^{N}}{3}\!\int_{\mathbb{R}^N}\left[f(w_0)w_y+f(w_y)w_0\right]dz = -\,\frac{2s^N\varepsilon_{\lambda}^{R}}{3} \leq -C_0\varepsilon_{\lambda}^{R}.
	\end{eqnarray*}
	Furthermore, it follows from \eqref{ACP-estimate} that
		\begin{eqnarray*}
			(III) + (IV) &\leq & o(\varepsilon_{\lambda}^R) -C_0 \varepsilon_{\lambda}^R. \end{eqnarray*}
All together, for $N \geq 3$, it holds
	\begin{equation}\label{5.13}
	I_V\!\left(U_{y,\lambda}^R\!\left(\frac{\cdot}{s}\right)\!\right) \leq 2p_0 - C_0\varepsilon_{\lambda}^{R} + o\!\left(\varepsilon_{\lambda}^{R}\right),
	\end{equation}
	for all $y \in \partial B_2(y_0)$, $\lambda \in [1/2-\delta,1/2+\delta]$ and $R \geq R_5$ sufficiently large. Most important is that $\lambda$ is in a closed sub-interval of $(0,1)$, so the bounds on $\varepsilon_{\lambda}^{R}$ are uniform in $\lambda$, which yields $\varepsilon_{\lambda}^{R} =O(R^{-(N-2)})$.

	On the other hand, for every $\lambda \in (0,1/2-\delta] \cup [1/2+\delta,1)$, $y \in \partial B_2(y_0)$ and $R \geq 1$ sufficiently large, if $s:=S_{y,\lambda}^{R} \leq 2$, then $\lambda s \in (0,1-2\delta]$ or $(1-\lambda)s \in (0,1-2\delta]$ and, if $s:=S_{y,\lambda}^{R} \geq 2$, then $\lambda s \in [1+2\delta,+\infty)$ or $(1-\lambda)s \in [1+2\delta,+\infty)$. In any case, either, $\lambda s \in (0,1-2\delta] \cup [1+2\delta,+\infty)$ or $(1-\lambda)s \in (0,1-2\delta] \cup [1+2\delta,+\infty)$. 
	
	Therefore, recalling that the map $t \mapsto I_{0}(w (\frac{\cdot}{t}))$ is strictly increasing in $(0,1]$ and strictly decreasing in $[1,\infty)$ and $I_0(w)=p_0$, there exist $\overline{\eta} \in (0,p_0)$ and $R_6$ sufficiently large, such that
	\begin{eqnarray*}
		(I) &=& I_{0}(w(\frac{\cdot}{\lambda s})) + I_{0}(w(\frac{\cdot}{(1-\lambda)s})) 
		\leq 2p_0 - 2\overline{\eta},
	\end{eqnarray*}
	for all $y \in \partial B_2(y_0)$, $\lambda \in (0, 1/2-\delta]\cup [1/2+\delta,1)$ and $R \geq R_6$. Hence, the previous estimates, imply that 
	\begin{eqnarray}\label{5.14}
	I_V\!\left(U_{y,\lambda}^R\!\left(\frac{\cdot}{s}\right)\!\right) \leq 2p_0 - 2\overline{\eta} + O(\varepsilon_{\lambda}^{R})\end{eqnarray}
	for all $y \in \partial B_2(y_0)$, $\lambda \in (0,1/2-\delta]\cup[1/2+\delta,1)$ and $R \geq R_6$.
	
	By Lemmas \ref{unic.s} and \ref{lema.extrem},  inequalities (\ref{5.13}) and (\ref{5.14}), taking $R_4:=\max\{R_0,R_3, R_5, R_6\}$, we get a number $\eta=\eta(R)> 0$ such that $$I_V\!\left(\!U_{y,\lambda}^R\!\left(\frac{\cdot}{s}\right)\!\right) \leq 2p_0-\eta,$$ for all $y \in \partial B_2(y_0)$, $\lambda \in [0,1]$ and $R \geq R_4$.
	
\end{proof}

For $c \in \mathbb{R}$, let us define $I_V^c:=\left\{u \in \mathcal{D}^{1,2}(\mathbb{R}^N): I_V(u) \leq c\right\}.$

Next we define a barycenter map that will be used in proving the existence of a solution of problem (\ref{1.1}).
Let $\beta: L^{2^*}(\mathbb{R}^N) \to \mathbb{R}^N$ be a barycenter function, i. e., a continuous map which satisfies $\beta(u(\cdot - y))= \beta(u) + y$ and $\beta (u \circ\theta 
^{-1})= \theta(\beta(u))$ for all $u \in L^{2^*}(\mathbb{R}^N)\setminus \{0\}$ and $y \in \mathbb{R}^N$, and every linear isometry $\theta$ of $\mathbb{R}^N$.
Note that $\beta (u)=0$ if $u$ is radial and $\beta(u(\cdot/s))=\beta(u)$ for $s>0$.

Now let us define 
\begin{equation}\label{def.b}
b:=\inf\!\left\{I_V(u): u \in \mathcal{P}_V, \beta(u)=0\right\}.
\end{equation}
Clearly, $b \geq p_V$.

\begin{lemma}\label{lema5.13}
	Assume {\bf ($V_1$)--($V_5$)} and {\bf ($f_1$)--($f_3$)} hold true. If $p_V$ is not attained by $I_V$ on $\mathcal{P}_V$, then $b > p_V$.
\end{lemma}

\begin{proof}
	Suppose, by contradiction, that $b = p_V$. Then, by definition, there exists a sequence $(v_n) \subset \mathcal{P}_V$, with $\beta(v_n)=0$, such that $I_V(v_n) \to b$. By Lemma \ref{lema3.2}, we have $(v_n)$ is bounded in $\mathcal{D}^{1,2}\!(\mathbb{R}^N)$. Using Ekeland's Variational Principle, we obtain a sequence $(u_n) \subset \mathcal{P}_V$ such that $I_V(u_n) \to p_V$ and $I_V|_{\mathcal{P}_V}'(u_n) \to 0$, with $\Vert u_n - v_n \Vert_V \to 0$, see \cite[Theorem 8.5]{Willem}. So by Lemma \ref{lema4.13}, we have $I'_V(u_n) \to 0$ in $(\mathcal{D}^{1,2}\!(\mathbb{R}^N))^{\prime}$. Since $(v_n)$ is bounded, it follows that $(u_n)$ is bounded in $\mathcal{D}^{1,2}\!(\mathbb{R}^N)$. Thus, if $p_V$ is not attained by $I_V$ on $\mathcal{P}_V$, it follows from Lemma \ref{lema4.15} that $u_n=w(\cdot - y_n) + o_n(1)$, where $o_n(1) \to 0$ as $n \to \infty$ and $(y_n) \subset \mathbb{R}^N$, $|y_n| \to +\infty$ and $w$ is the radial solution of problem (\ref{P_0}). Doing a translation, we get $u_n(x+y_n)=w(x)+o_n(1)$. Using the barycenter function, we obtain $\beta(u_n(x+y_n))=\beta(u_n)-y_n=-y_n$ and $\beta(w(x)+o_n(1))=\beta(w(x))+o_n(1)$, by the continuity. Since $w$ is radial, it follows that $\beta(w(x)) = 0$ and so $-y_n = o_n(1)$, which is a contradiction. Therefore, $b > p_V$. 
\end{proof}

\begin{lemma}\label{lema5.14}
	Assume that {\bf ($V_1$)--($V_5$)} and {\bf ($f_1$)--($f_3$)} hold true. If $p_V$ is not attained by $I_V$ on $\mathcal{P}_V$, then $p_V = p_0$ and there exists $\delta > 0$ such that $$\beta(u) \neq 0, \quad \forall \, u \in \mathcal{P}_V \cap I_V^{p_0+\delta}.$$
\end{lemma}

\begin{proof}
	By Lemma \ref{lema4.12}, we have $p_V \leq p_0$. On the other hand, it follows from Lemma \ref{lema4.16} that, if $p_V$ is not attained by $I_V$ on $\mathcal{P}_V$, then $p_V \geq p_0$ and so $p_V = p_0$. Now let us show that there exists $\delta > 0$ such that $$\beta(u) \neq 0, \quad \forall \, u \in \mathcal{P}_V \cap I_V^{p_0+\delta}.$$ Suppose, by contradiction, that for all $n \in \mathbb{N}$ there exists $v_n \in \mathcal{P}_V$ such that $I_V(v_n) \leq p_0 + 1/n$ and $\beta(v_n) = 0$. Thus, we have $b \leq I_V(v_n) \leq p_0 + 1/n$, for all $n \in \mathbb{N}$. So as $n \to \infty$, it follows that $b \leq p_0 = p_V$, contradicting Lemma \ref{lema5.13}. Therefore, the result follows. 
\end{proof}

{\bf Proof of Theorem \ref{teo1.1}.} If $p_V$ is attained by $I_V$ at some $u \in \mathcal{P}_V$ then, by Corollary \ref{cor4.14}, $u$ is a nontrivial solution of problem (\ref{1.1}). So assume that $p_V$ is not attained. Then, using Lemmas \ref{lema4.12} and \ref{lema4.16}, it follows that $p_V = p_0$. We will show that $I_V$ has a critical value in $(p_0,2p_0)$. 

Lemma \ref{lema5.14} allows us to choose $\delta \in (0,p_0/4)$ such that $\beta(u) \neq 0, \forall \, u \in \mathcal{P}_V \cap I_V^{p_0+\delta}$ and, by Lemma \ref{lema.extrem} and Proposition \ref{prop.princ}, we may choose $R \geq 1$ and $\eta  \in (0,p_0/4)$ such that $$I_V\!\left(\!U_{y,\lambda}^R\!\left(\frac{\cdot}{s}\right)\!\right) \leq \left\{\begin{array}{ll}
p_0 + \delta, & \ \ \textrm{for} \ \lambda = 0 \ \textrm{and all} \ y \in \partial B_2(y_0), \\
2p_0 - \eta, & \ \ \textrm{for all} \ \lambda \in [0,1] \ \textrm{and all} \ y \in \partial B_2(y_0), \\
\end{array}\right.$$ where $s:=S_{y,\lambda}^{R}>0$ is such that $U_{y,\lambda}^R\!\left(\frac{\cdot}{s}\right) \in \mathcal{P}_V$. Define $\zeta: B_2(y_0) \to I_V^{2p_0-\eta}$ by $$\zeta(\lambda y_0 + (1-\lambda)y):=U_{y,\lambda}^R\!\left(\frac{\cdot}{s}\right), \quad \textrm{with} \, \lambda \in [0,1], \, y \in \partial B_2(y_0).$$
 Arguing by contradiction, assume that $I_V$ does not have a critical value in $(p_0,2p_0)$. Then, there exists $\varepsilon > 0$ such that $\Vert I_V'(u) \Vert_{\left(\mathcal{D}^{1,2}(\mathbb{R}^N)\right)'} \geq \varepsilon,  \forall \, u \in I_V^{-1}\!\left([p_0+\delta,2p_0-\eta]\right).$ Otherwise there would be $d \in (p_0,2p_0)$ and a sequence $(u_n) \in \mathcal{P}_V$ such that $I_V(u_n) \to d$, $I_V|_{\mathcal{P}_V}'(u_n) \to 0$ and, so Corollaries \ref{cor4.17} and \ref{cor4.14}, would lead to a contradiction. Then, there exists a continuous function $\pi:\mathcal{P}_V \cap I_V^{2p_0-\eta} \to \mathcal{P}_V \cap I_V^{p_0+\delta}$ such that $\pi(u) = u$ for all $u \in \mathcal{P}_V \cap I_V^{p_0+\delta}$, see \cite[Lemma 5.15]{Willem}. Note that the function $h:B_2(y_0) \to \partial B_2(y_0)$, given by $$h(x):=2\left(\frac{(\beta \circ \pi \circ \zeta)(x) - Ry_0}{|(\beta \circ \pi \circ \zeta)(x) - Ry_0|}\right) + y_0,$$ is well defined and continuous. Moreover, if $y \in \partial B_2(y_0)$, then $\zeta(y) = U_{y,0}^R\!\left(\frac{\cdot}{s}\right) \in I_V^{p_0+\delta}$, with $(\beta \circ \pi \circ \zeta)(y) = \beta\!\left(U_{y,0}^R\!\left(\frac{\cdot}{s}\right)\!\right) = \beta\!\left(w\!\left(\frac{\cdot}{s}-Ry\right)\!\right) = \beta\!\left(w\!\left(\frac{\cdot}{s}\right)\!\right) + Ry = Ry$ and, hence, $h(y)=y$ for every $y \in \partial B_2(y_0)$. So we get the following restriction map $\tilde{h}:=h|_{\partial B_2(y_0)}:\partial B_2(y_0) \to \partial B_2(y_0)$, given by $\tilde{h}(y)=y$. But the existence of such a contract $h$ contradicts Brouwer's Fixed Point Theorem. Therefore, $I_V$ must have a critical point $u \in \mathcal{P}_V$, with $I_V(u) \in (p_0,2p_0)$. This proves that problem (\ref{1.1}) has a nontrivial solution $u \in \mathcal{D}^{1,2}(\mathbb{R}^N)$. Using the maximum principle we can conclude that $u$ is positive and the proof of the theorem is complete. 
\begin{flushright}
	$\Box$
\end{flushright}


\end{document}